\documentclass[]{article}

\usepackage{amsmath,amssymb,exscale,epsfig}

\usepackage{delarray}
\usepackage{amsmath}    
\usepackage{amssymb, bbm}    
\usepackage{amsfonts}   
\usepackage{graphics}
\usepackage{graphicx}
\usepackage{stmaryrd}

\usepackage{amsopn}
\usepackage[ansinew]{inputenc}
\usepackage{tabularx}
\usepackage{xcolor}
\usepackage[numbers]{natbib}

\usepackage{lineno}
\usepackage{pdfcomment}
\usepackage[normalem]{ulem}

\usepackage{authblk}

\newtheorem{theoreme}{Theorem}

\newtheorem{lemme}{Lemma}
\newtheorem{definition}{Definition}
\newtheorem{proposition}{Proposition}
\newtheorem{corollaire}{Corollary}
\newtheorem{remarque}{Remark}

\newcommand{\T}{\ensuremath{{\mathcal{T}}}}

\newcommand{\B}{\ensuremath{{\mathcal{B}}}}
\newcommand{\A}{\ensuremath{{\mathcal{A}}}}

\newcommand{\NNN}{\ensuremath{{\mathcal{N}}}}
\newcommand{\G}{\ensuremath{\mathcal{G}}}

\renewcommand{\P}{\ensuremath{{\mathcal{P}}}}
\newcommand{\I}{\ensuremath{{\mathcal{I}}}}

\newcommand{\M}{\ensuremath{\mathcal{M}}}

\newcommand{\HH}{\ensuremath{{\mathcal{H}}}}
\newcommand{\R}{\ensuremath{{\mathbb{R}}}}

\newcommand{\Rd}{\ensuremath{{\mathbb{R}^d}}}

\newcommand{\K}{\ensuremath{\mathcal{K}}}

\newcommand{\Z}{\ensuremath{\mathbb{Z}}}
\newcommand{\Zd}{\ensuremath{{\mathbb{Z}^d}}}

\newcommand{\om}{\ensuremath{\omega}}
\newcommand{\Oi}{\ensuremath{\Omega}_\infty}
\renewcommand{\t}{\ensuremath{\theta}}
\renewcommand{\T}{\ensuremath{\Theta}}
\renewcommand{\O}{\ensuremath{\Omega}}

\newcommand{\1}{\ensuremath{\mbox{\rm 1\kern-0.23em I}}}
\renewcommand{\L}{\ensuremath{\Lambda}}
\renewcommand{\l}{\ensuremath{\lambda}}
\newcommand{\eps}{\ensuremath{\varepsilon}}

\usepackage{authblk}

\begin{document}

\title{Consistency of likelihood estimation for Gibbs point processes}

\author[1]{David Dereudre}
\author[2,3]{Fr\'ed\'eric Lavancier}
\affil[1]{Laboratoire de Math\'ematiques Paul Painlev\'e\\University of Lille 1, France}
\affil[2]{Laboratoire de Math\'ematiques Jean Leray\\
University of Nantes, France}
 \affil[3]{Inria, Centre Rennes  Bretagne Atlantique, France.}

\maketitle

\begin{abstract}

Strong consistency of the maximum likelihood estimator (MLE) for parametric Gibbs point process models is established.
The setting is very general.  It includes pairwise pair potentials, finite and infinite multibody interactions and geometrical interactions, where the range can be finite or infinite. The Gibbs interaction may depend linearly or non-linearly on the parameters, a particular case being  hardcore parameters and interaction range parameters. 
As important examples,  we deduce the consistency of the MLE for all parameters  of the Strauss model, the hardcore Strauss model, the Lennard-Jones model  and the area-interaction model.

  \bigskip


\end{abstract}

\section{Introduction}

Gibbs point processes are popular and widely used models in spatial statistics to describe the repartition of points or geometrical  structures in space. They initially arose from  statistical physics  where they are models for  interacting continuum particles, see for instance \cite{B-Pre76}. 
They are now used in as different domains as astronomy, biology, computer science, ecology, forestry,  image analysis and materials science.  
The main reason is that  Gibbs point processes provide a clear interpretation of the interactions between the points, such as attraction or repulsion depending on their relative position. We refer to \cite{chiu2013}, \cite{B-IllPenSto08}, \cite{B-MolWaa03} and \cite{vanlieshout2000} for 
classical text books on spatial statistics and stochastic geometry, including examples and applications of Gibbs point processes.

Assuming a parametric form of the Gibbs interaction, the natural method to estimate the parameters is likelihood inference.
A practical issue however is that the likelihood depends on  an intractable normalizing constant, called the partition function in the statistical physics literature, that has to be approximated. Some sparse data approximations were first proposed  in the 80's, e.g. in \cite{ogata1984}, before simulation-based methods have been developed \cite{geyer-moller1994}.  A comparative simulation study carried out in \cite{mateu_montes2001} demonstrates that Monte Carlo approximation of the partition function provides the best results in practice. To avoid the latter approximation, other estimation methods have been introduced, including pseudolikelihood and moments based methods, see the  books cited above. 
With modern computers, the Monte-Carlo approximation of the partition function is no longer an important issue and maximum likelihood estimation for spatial data is feasible and widely used in practice.

From a theoretical point of view, very few is known about the asymptotic properties of the maximum likelihood estimator (MLE) for Gibbs point processes. The asymptotic here means that the window containing the point pattern increases to the whole space $\R^d$. It is  commonly believed that the MLE is consistent and more efficient, at least asymptotically, than the other estimation methods.  This conviction has been supported by several simulation studies, ever since \cite{ogata1984}, see also \cite{diggle1994}. The present work is concerned with  the basic question of consistency. The latter is conjectured to hold in a very general setting but no proof were so far available in the continuous case of Gibbs point processes on $\R^d$. 
This is in contrast with the discrete case of  Gibbs interactions on a lattice, where consistency is established for most standard parametric models, regardless of the occurrence of phase transition (when the Gibbs measure is not unique), see \cite{kunsch1981} and \cite{Guyon}. In fact, the continuous case  is more  challenging in that most parametric Gibbs models involve irregular parameters as hardcore parameters (controlling the support of the measure) or interaction range parameters.  These specificities  result in a discontinuous likelihood contrast function, even asymptotically, and some further technical difficulties (for instance the true unknown Gibbs measure is not absolutely continuous with respect to the Gibbs measure associated to an estimation of the hardcore parameter).  
In an unpublished manuscript   \cite{mase02}, S.~Mase addressed  the consistency of the MLE for superstable and regular  pairwise  interactions  (a formal definition will be given later). His main tool was the variational principle for Gibbs processes, following the initial idea  developed in the discrete case in \cite{kunsch1981}. He restricted his study to pairwise interactions that are linear in their parameters,  which yields a convex contrast function. His result do not imply any restriction on the parameter space, thus including the possibility of  phase transition, but the parametric interactions considered in \cite{mase02} remain nonetheless rather restrictive and do not include hardcore or interaction range parameters.

We prove in this paper that the MLE is strongly consistent for a wide class of stationary Gibbs interactions, without any  major restriction on the parameter space. Our assumptions include finite and infinite-body interactions  with finite or infinite range, and we do not assume any continuity with respect to the parameters.  Our result covers in particular the consistency of the MLE of all parameters  of the Strauss model (including the range of interaction), the hardcore Strauss model (including the hardcore parameter), the Lennard-Jones model (including the exponent parameters), and the area-interaction process (including the radius of balls). An important ingredient of the proof is the variational principle, as in  \cite{kunsch1981} and \cite{mase02}, which guarantees the identifiability of the parameters. Our original contribution is the formulation of a minimum contrast result in presence of discontinuities and,  in order to apply it, new controls of physical quantities as the pressure and the mean energy with respect to the parameters.

Beyond consistency, the next natural question concerns the asymptotic distribution of the MLE.  This problem is even more arduous and we do not address it in the present paper. Nonetheless, let us briefly mention the state of the art on this question. In 1992, S.~Mase \cite{mase92}  proved that the MLE of regular parameters in a certain class of Gibbs interactions is uniform locally asymptotic normal. However his result, relying on strong cluster estimates,  is established for Gibbs measure generating very sparse point patterns, which implies restrictive conditions on the parameter space. In the same period, J.~Jensen \cite{jensen1993}  proved the asymptotic normality of the MLE under the Dobrushin uniqueness region, where the Gibbs process satisfies mixing properties. Here again,  as noticed in \cite{geyer-moller1994},  this assumption implies strong restrictions on the parameter space. Without these conditions, phase transition may occur  and some long-range dependence phenomena can appear. The MLE might then exhibit a non standard asymptotic behavior, in the sense that the rate of convergence might differ from the standard square root of the size of the window and the limiting law might be non-gaussian. 
To our knowledge, the only well understood example is the estimation of the inverse temperature in the Ising model on a square lattice studied in \cite{pickard87}. 
For this example, the asymptotic law of the MLE is always Gaussian and  the rate of convergence is standard except at  the critical temperature where it is faster.  Generalizing this result to other Gibbs models, especially in the continuous case,  is hazardous. The main reason is that the occurence of phase transition is in general not well understood, and this phenomenon can be of very different nature depending on the model. 
Characterizing the asymptotic distribution of the MLE in a general setting of Gibbs models still remain a challenging open question.

 The remainder of the paper is organized as follows. Parametric Gibbs point processes and the MLE procedure are described in Section~\ref{sec1}. Section~\ref{general section} contains our main result, namely minimal conditions on the Gibbs interaction to ensure strong consistency of the MLE. While the latter result is established in a very general setting, we present in Section~\ref{pair 	interaction} standard families of models where our main result applies. Specifically, we deal with  finite range pair potentials with or without hardcore (including the Strauss model and the hardcore Strauss model), infinite range  pair potentials (including the Lennard-Jones model), and infinite-body interactions (including the area-interaction model). Section~\ref{proofs} contains the proofs of our results.

\section{Gibbs point processes and the MLE}\label{sec1}

\subsection{ State space, reference measure and notation}

We consider the continuous space $\Rd$ of arbitrary dimension $d \ge 1$.   The 
Lebesgue measure on $\Rd$ is denoted by $\l^d$ and the symbol $\L$ will always refer to a bounded Borel subset of $\R^d$.
For $x\in\R^d$, $|x|$ denotes the Euclidean norm of $x$ while for $\L\subset\R^d$, $|\L|:=\l^d(\L)$.

A {\it configuration} is a subset $\om$ of $\Rd$ 
which is locally finite, meaning that $\om\cap\L$ has finite cardinality
$N_\L(\om):=\#(\om\cap\L)$ for every bounded Borel set $\L$. 
The space $\O$ of all configurations is equipped with the 
$\sigma$-algebra $\mathcal{F}$ generated by the counting variables $N_\L$.

For convenience, we  will often write $\om_\L$ in place of $\om\cap\L$  and $\om_{\L^c}$ for $\om\cap\L^c=\om\setminus\om_\L$. Similarly for every $\om$ and every $x$ in $\om$
we abbreviate $\om\cup\{x\}$ to $\om\cup x$ 
and   $\om\backslash\{x\}$ to $\om\backslash x$.

As usual, we take the reference measure on $(\O,\mathcal{F})$ 
to be the distribution $\pi$ of the Poisson point process 
with intensity measure $\l^d$  on $\Rd$.
Recall that $\pi$ is the unique probability measure
on $(\O,\mathcal{F})$ such that the following hold for all subsets $\L$:
(i) $N_{\L}$ is Poisson distributed with parameter $\l^d(\L)$, 
and (ii) given $N_\L=n$, 
the $n$ points in $\L$ are independent with uniform distribution on $\L$. 
The Poisson point process restricted to $\L$ will be denoted by $\pi_\L$.

Translation by a vector $u \in \Rd$ is denoted by $\tau_u$,
either acting on $\Rd$ or on $\O$. 
A probability $P$ on $\O$ is said stationary if  $P=P\circ\tau_u^{-1}$ for any $u$ in $\Rd$.  In this paper we consider only stationary probability measures $P$  with finite intensity $E_P(N_{[0,1]^d})$, where $E_P$ stands for the expectation with respect to $P$.  We denote by $\P$ the space of such probability measures. 

We denote by $\Delta_0$, $\I_n$ and $\L_n$ the following sets 
\begin{equation*}\label{sets} \Delta_0=[0,1)^d, \quad \I_n=\{-n,-n+1,\ldots, n-1\}^d \quad \text{ and } \L_n=\bigcup_{k\in\I_n} \tau_k(\Delta_0)=[-n,n)^d.\end{equation*}

In the following, some infinite range interaction processes will be considered. To ensure their existence, we must restrict the set of configurations to the so-called set of tempered configurations as in \cite{Ruelle70}. The definition of the latter may depend on the type of interactions at hand. In case of superstable pairwise interactions, it is simply defined by
$$\Omega_T=\{\omega\in\Omega;\; \exists t>0, \forall n\geq 1, \sum_{i\in \I_n} N^2_{ \tau_i(\Delta_0)}(\omega) \leq t (2n)^d\}.$$
Unless specified otherwise, this is the definition we consider in the sequel. From the ergodic theorem, any second order stationary measure on $\Omega$ is supported on $\Omega_T$, so the restriction to $\Omega_T$ is a mild assumption in practice.  

 Several other notation are introduced throughout the next sections. For convenience to the reader, we summarize the most important of them below.

\begin{tabular}{cl}
$\O$, $\O_T$ & space of configurations, tempered configurations respectively\\
$\omega_\L$, $\omega_{\L^c}$  & configuration $\om$ inside $\L$, outside $\L$ respectively\\
$\delta$, $\theta$ & parameters of the interaction where $\delta$ is the hardcore parameter\\
$I$, $\Theta$, $\mathcal K$ &  parameters space: $\delta\in I=[\delta_{\min},\delta_{\max}]$, $\theta\in\Theta$ and $\mathcal K\subset\Theta$ is compact\\
$\Omega^{\delta}_\infty$ & space of configurations with locally finite energy\\
$H^{\t}_\L (\om)$ & energy (or Hamiltonian) of $\omega$ in $\Lambda$\\
$f^{\delta,\t}_\L(\om)$  & conditional density of $\om$ inside $\Lambda$, given $\om_{\Lambda^c}$, see \eqref{parametric localdensity}\\
$Z^{\delta,\t}_\L$ & partition function associated to the free boundary condition, see \eqref{partition}\\
$\G^{\delta,\t}$ & set of stationary Gibbs measures associated to $\om\mapsto f^{\delta,\t}_\L(\om)$\\
$\G$ & $\G=\cup_{\delta\in I, \theta\in\T} \G^{\delta,\t}$\\
$\om^*$ & configuration associated to the true unknown parameters $\delta^*$ and $\t^*$\\
$K_n^{\delta,\t}(\om^*_{\L_n})$ & likelihood contrast function for the observation of $\om^*$ on $\Lambda_n$, see \eqref{contrastf}\\
$H^\t(P)$, $\mathcal I(P)$ & mean energy and specific entropy of $P\in\G$, see \eqref{def mean energy} and \eqref{defentropy}\\
$p(\delta,\theta)$ & pressure associated to $P^{\delta,\theta}\in\G^{\delta,\t}$, see \eqref{defpressure} 
\end{tabular}

\subsection{Gibbs point processes}
From a general point of view,  a  {\it family of interaction energies} is a collection $\HH = (H_\L)$,
indexed by bounded Borel sets $\L$, 
of measurable functions from $\O_T$ to $\R\cup\{+\infty\}$ such that
for every $\L \subset \L'$, 
there exists a measurable function 
$\varphi_{\L,\L'}$ from $\O_T$ to $\R\cup\{+\infty\}$ 
and  for every $\om\in\O_T$
\begin{equation}\label{compatible}
 H_{\L'}(\om)= H_{\L}(\om) + \varphi_{\L,\L'}(\om_{\L^c}). 
\end{equation}
 This decomposition  
is equivalent to (6.11) and (6.12) in \cite[p.\ 92]{B-Pre76}. In physical terms, $H_\L(\om)=H_\L(\om_\L\cup\om_{\L^c})$ represents 
the potential energy of the configuration $\om_\L$  inside $\L$ given 
the configuration $\om_{\L^c}$ outside $\L$. 
 In words, \eqref{compatible} is just a compatibility property stating that the difference between the energy on $\Lambda'$ and the energy on a subset $\Lambda$ only depends on the exterior configuration $\om_{\L^c}$.

As explained in introduction, we aim at  including  a large class of interactions in our study, to cover for instance infinite-body interactions like  the area-interaction process considered in Section~\ref{infinite body}. 
Nevertheless, most of standard parametric Gibbs models are pairwise interaction point processes. They are introduced below and we will come back to this important class of Gibbs models  in Section~\ref{pair interaction}.
 {\it Pairwise potential interactions} take the particular form 
\begin{equation}\label{defH}
H_\L (\om)=z\ N_\L(\om) + \sum_{\{x,y\}\in\om, \{x,y\}\cap \om_\L\neq\emptyset} \phi(x-y),\end{equation}
where $z>0$ is the intensity parameter and $\phi$ is the pair potential,  a function from $\R^d$ to  $\R\cup\{+\infty\}$ which is symmetric, i.e. $\phi(-x)=\phi(x)$ for all $x$ in $\R^d$.  In connection with \eqref{compatible}, if $H_\L$ satisfies \eqref{defH}, then $\varphi_{\L,\L'}(\om_{\L^c})=z\ N_{\L'}(\om_{\L^c}) + \sum_{\{x,y\}\in\om_{\L^c}, \{x,y\}\cap \L'\neq\emptyset} \phi(x-y)$.

Let us present two well-known examples.  We will use them through the paper to illustrate our notation and assumptions.\\

\noindent{\it Example 1}:
The Strauss pair potential, defined for some possible hardcore parameter $\delta\geq 0$, some interacting parameter $\beta\geq 0$, and some range of interaction $R\geq 0$ corresponds in \eqref{defH} to  
\begin{equation}\label{strauss}
\phi(x)=\begin{cases} \infty  & \text{if}\quad  |x|<\delta,\\ \beta & \text{if}\quad  \delta\leq |x|<R, \\0 & \text{if}\quad |x|\geq R,\end{cases}
\end{equation}
if $R>\delta$, while $\phi(x)= \infty\1_{  |x|<\delta}$ if $R\leq\delta$. As we will see later, the theoretical properties of this model strongly differ whether   $\delta=0$, which corresponds to the standard Strauss model, or $\delta>0$, which is the hardcore Strauss model.\\

\medskip

\noindent{\it Example 2}:
The general Lennard-Jones $(n,m)$ pair potential is defined for some $d<m<n$, and some $A>0$, $B\in \R$ by  \eqref{defH} with
\begin{equation}\label{lennard}
\phi(x)=A|x|^{-n}-B|x|^{-m},\quad x\in\R^d.\end{equation}
The standard Lennard-Jones model in dimension $d=2$ and $d=3$  corresponds to $n=12$ and $m=6$.\\

For a family of energies $\HH = (H_\L)$, we denote by $\Oi$ 
the space of configurations which have a locally finite energy, i.e.\ $\om\in\Oi$ if and only if, for any bounded Borel set $\L$,  $H_\L (\om)$ is finite.

The Gibbs measures $P$ associated to $\HH$ are defined through their local conditional specification, as described below. For every $\L$ and every $\om$ in $\Oi\cap \O_T$, the  conditional density  $f_\L$ of $P$ with respect to $\pi_\L$ is defined by
\begin{equation}
\label{localdensity}
f_\L(\om)  =  \frac{1}{Z_\L(\om_{\L^c})} e^{-H_\L (\om)},
\end{equation} 
where $ Z_\L(\om_{\L^c})$ is the normalization constant, or partition function,  given by 
$$Z_\L(\om_{\L^c})= \int e^{-H_\L(\om'_\L \cup \om_{\L^c})}\pi_\L(d\om'_\L).$$
Some regularity assumptions on $\HH$ are required to ensure that $0< Z_\L(\om_{\L^c})< +\infty$, implying  that the local density is well-defined. They will be part of our general hypothesis in Section~\ref{general section} and are fulfilled for all our examples in Section~\ref{pair interaction}.

We are now in position to define the Gibbs measures associated to $\HH$, see for instance  \cite{georgii}.
\begin{definition}
A probability measure $P$ on $\O$ is a {\it Gibbs measure}
for the family of energies $\HH$ 
if $P(\Oi\cap\O_T)=1$ and, for every bounded borel set $\L$, 
for any measurable and integrable function $g$ from $\O$ to $\R$,
\begin{equation}\label{DLR}
\int g(\om)P(d\om) = \int \int g(\om'_\L \cup \om_{\L^c}) f_\L(\om'_\L \cup \om_{\L^c}) \pi_\L({\rm d}\om'_\L) P(d\om).
\end{equation} 
Equivalently, 
for $P$-almost every $\om$ the conditional law of $P$ given $\om_{\L^c}$ 
is absolutely continuous with respect to $\pi_\L$ with the density $f_\L$ defined in \eqref{localdensity}.
\end{definition}

The equations (\ref{DLR}) are called the Dobrushin--Lanford--Ruelle (DLR)
equations.  
Conditions on $\HH$ are mandatory to ensure the existence of a measure $P$ satisfying \eqref{DLR}. For general interactions, we will later assume that both \eqref{localdensity} is well-defined and $P$ exists.  On the other hand, the unicity of $P$ does not necessarily hold, leading to  phase transition. As explained in introduction, our consistency results are not affected by this phenomenon.

\subsection{Parametric Gibbs models and the MLE procedure}\label{param Gibbs}
We consider a parametric Gibbs model that depends on a hardcore parameter $\delta\in I$, where $I=[\delta_{\min},\delta_{\max}]$ with $0\leq \delta_{\min}\leq \delta_{\max}  \leq \infty$, and on a parameter $\t\in\T$ with $\T\subset\R^p$ and  $\Theta\neq\emptyset$. We stress the dependence on these parameters by adding some superscripts to our notation, e.g. the partition function $Z_\L(\om_{\L^c})$ becomes $Z^{\delta,\t}_\L(\om_{\L^c})$ and the interaction energy $H_\L(\om)$ becomes $H^\t_\L(\om)$. In particular, the reason why we write  $H^\t_\L(\om)$ and not $H^{\delta,\t}_\L(\om)$ (and similarly for other quantities) is due to the fact that under our assumptions, the interaction energy will not depend on $\delta$,  as explained below.  

Specifically we assume that the conditional density writes for any $\delta\in I$ and any $\t\in\T$
\begin{equation}\label{parametric localdensity}
f^{\delta,\t}_\L(\om)  =  \frac{1}{Z^{\delta,\t}_\L(\om_{\L^c})} e^{-H^{\t}_\L (\om)}  \1_{\Oi^\delta}(\om),\quad \om\in\O_T,
\end{equation}
where 
$$\om\in\Oi^\delta  \Leftrightarrow \inf_{\{x,y\}\in\om} |x-y| \geq \delta,$$ 
while for any  $\theta\in\T$ and any $\om\in\O_T$, $H^{\t}_\L (\om)<\infty$. 
This specific parametric form clearly indicates that  the hardcore parameter $\delta$ only rules the support $\Oi^\delta$ of the measure, and  has no effect on  the interaction energy $H^\t_\L$. This is what happens for most models in spatial statistics, see Section~\ref{pair 	interaction}. A counter-example is  the Diggle-Gratton model \cite[Section 6]{diggle1984}, where the interaction energy depends also on $\delta$. This situation is not covered by our study.

Let $\G^{\delta,\t}$ be the set of stationary Gibbs measures  defined by the conditional density  \eqref{parametric localdensity} and with  finite intensity.  A minimal condition is to assume that for any value of the parameters, \eqref{parametric localdensity} is well-defined and that the latter set is not empty. \\

 \noindent {\bf [Existence]}:  For any  $\delta\in I$,  any  $\t\in\T$,  any bounded set $\L$ in $\R^d$ and any  $\om\in\Oi^\delta\cap\O_T$, 
 $Z_\L^{\delta,\theta}(\om_{\L^c})<\infty$ and  the set $\G^{\delta,\t}$ is not empty. \\

\noindent{\it Example 1 (continued)}: 
 For the hardcore Strauss model defined in \eqref{strauss} with $\delta>0$, the hardcore parameter is $\delta$ and the parameter $\theta$ corresponds to $\theta=(z,R,\beta)$. The assumption {\bf [Existence]} for this model holds if we take $I=\R_+$ and  $\T=\R_+\times\R_+\times \R$. On the other hand, for the standard Strauss model corresponding to the case $\delta=0$  in \eqref{strauss}, there is no hardcore parameter, meaning that $I=\{0\}$ and  the only parameter of the model is $\theta=(z,R,\beta)$. The existence of this model, i.e.  {\bf [Existence]}, holds iff $\theta$ belongs to $\T=\R_+\times\R_+\times \R_+$. Note that $\beta$ needs to be nonnegative when there is no hardcore, contrary to the hardcore case. We refer to \cite{Ruelle70} for the existence results involving stable and superstable pairwise potentials, a class of interactions which includes the Strauss model. 
 
 \medskip

 \noindent{\it Example 2 (continued)}: For the Lennard-Jones model defined in \eqref{lennard}, there is no hardcore parameter ($\delta=0$) and the parameter $\theta$ corresponds to $(z,A,B,n,m)\in \Theta\subset \R^5$ where $\Theta$ is just defined via the constraints $z\ge 0$, $A>0$ and $d<m<n$. The existence is also proved in \cite{Ruelle70}.\\

 Let $\delta^*\in  I$ and $\t^*$ in $\T$. We denote by $\om^*$ a realization of a Gibbs measure belonging to $\G^{\delta^*,\t^*}$.
The parameters $\delta^*$ and $\t^*$ represent the true unknown parameters that we want to estimate. 
 The MLE  of $(\delta^*,\t^*)$ from the observation of $\om^*$ on $\L_n$ is defined as follows.


\begin{definition}
Let $\mathcal K$ be a compact subset of $\T$ such that $\t^*\in\mathcal K$.
The MLE of $\delta^*$ and $\t^*$ from the observation of $\om^*$ in $\L_n$ is defined by
\begin{equation}\label{def MLE}
(\hat\delta_n, \hat\t_n)=\underset{(\delta,\t)\in I\times\mathcal K}{\emph{argmax}} f^{\delta,\t}_{\L_n}(\om^*_{\L_n}).
\end{equation}
\end{definition}

In Lemma \ref{MLEexistence} below, we give sufficient conditions which ensure the existence of the argmax in \eqref{def MLE}. Let us note that we consider here the MLE with free boundary condition, meaning that the configuration $\om^*_{\L^c_n}$ outside $\L_n$ is not involved. This is the most natural setting given that we observe $\om^*$ only on $\L_n$. A MLE procedure that depends on the outside configuration could have been considered as well and similar theoretical results would have been proved. To implement this alternative procedure in practice yet, the interaction has to be finite range so that the outside configuration reduces  to boundary effects that can be handled by minus sampling.  For  infinite range potentials, as considered in Section \ref{sectionLJ},  this method is not feasible and  the MLE with free boundary condition makes more sense. 
 
\medskip

We will use in the following the equivalent definition of the MLE
\begin{equation*}\label{contrast function}
(\hat\delta_n, \hat\t_n)=\underset{(\delta,\t)\in I\times\mathcal K}{\text{argmin}} K_n^{\delta,\t}(\om^*_{\L_n}),
\end{equation*}
where $K_n^{\delta,\t}$ is the contrast function
\begin{equation}\label{contrastf}
K_n^{\delta,\t}(\om^*_{\L_n})= \frac{\ln(Z^{\delta,\t}_{\L_n})}{|\L_n|}+ \frac{H^{\t}_{\L_n}(\om^*_{\L_n})}{|\L_n|} + \infty \1_{\tilde\delta_n (\om^*_{\L_n})<\delta}
\end{equation}
with  \begin{equation}\label{delta tilde}\tilde\delta_n (\om_{\L_n}) = \min_{\{x,y\}\in\om_{\L_n}} |x-y|\end{equation}
and
\begin{equation}\label{partition}Z^{\delta,\t}_\L =Z^{\delta,\t}_\L(\emptyset) =  \int e^{-H_\L^{\t}(\om_\L)} \1_{\Oi^\delta}(\om_\L)\pi_\L(d\om_\L).\end{equation}
In statistical mechanics, the first term in \eqref{contrastf} is called the finite volume pressure while the second term corresponds to the  specific energy.

In the following lemma we give an explicit expression for $\hat\delta_n$. 
\begin{lemme}\label{MLE hardcore}
Under   {\bf [Existence]},  the MLE of the hardcore parameter $\delta$ is 
$\hat\delta_n = \tilde\delta_n$ if $\tilde\delta_n\in I$ and $\hat\delta_n = \delta_{\max}$ otherwise, where $\tilde\delta_n$ is given by \eqref{delta tilde}, consequently
\begin{equation}\label{def hat theta}
 \hat\t_n=\underset{\theta\in \mathcal K}{\emph{argmin}}\ K_n^{\hat\delta_n,\t}(\om^*_{\L_n}).
 \end{equation}
\end{lemme}

\begin{proof}
 If $\delta>\tilde\delta_n$, then $K_n^{\delta,\theta}(\om^*_{\L_n})=\infty$ and so $\hat \delta_n\leq\tilde\delta_n$. For any $\delta\leq\tilde\delta_n$ and any $\t\in\T$, $K_n^{\delta,\theta}(\om^*_{\L_n})=|\L_n|^{-1} (\ln(Z^{\delta,\t}_{\L_n}) + H^{\t}_{\L_n}(\om^*_{\L_n}))$. Note that if $\delta<\delta'$,  then $\Oi^{\delta}\subset \Oi^{\delta'}$ and from \eqref{partition}  the function $\delta\mapsto Z^{\delta,\t}_{\L_n}$ is decreasing.
Therefore, for any $\t\in\T$ and  for any $\delta\leq\tilde\delta_n$, $\delta\mapsto K_n^{\delta,\theta}(\om^*)$ is decreasing, proving that  $\hat \delta_n\geq\tilde\delta_n$. Hence $\hat \delta_n = \tilde\delta_n$ if $\tilde\delta_n\in I$.  The other  statements of Lemma~\ref{MLE hardcore} are straightforward.
\end{proof}\\

This lemma provides a useful result for practical purposes, since it states that the MLE for the hardcore parameter is just the minimal distance between pairs of points observed in the pattern, as it is usually implemented. Let us remark that a crucial assumption is that the space of parameters in {\bf [Existence]} is the cartesian  product $ I\times\T$, i.e. the parametric model is well-defined for any value of $\t\in\Theta$, whatever  $\delta$ is. This is the case for the examples considered in Section~\ref{pair interaction}. 

As explained in introduction, the computation of $\hat\t_n$ given by \eqref{def hat theta} is more difficult because there is 
in general no closed form expression for $Z^{\delta,\t}_\L$. In practice the optimization \eqref{def hat theta} is conducted from an approximated contrast function, based on Markov chain Monte Carlo methods where the error of approximation can be controlled, see for instance \cite{geyer-moller1994}.
Even to ensure the existence of $\hat \t_n$, we need some extra assumptions, as detailed in the following lemma.\\

\noindent {\bf [Argmax]}: 
For any $\om\in\O_T$ and any $n\ge 1$,  the function $\t\mapsto H^\t_{\L_n}(\om_{\L_n})$ is lower semicontinuous  over $\mathcal K$, i.e. for any $\t\in\mathcal K$ $\liminf_{\t'\mapsto \t} H^{\t'}_{\L_n}(\om_{\L_n})\ge H^\t_{\L_n}(\om_{\L_n})$,
 and there exists an upper semicontinuous version  $\t\mapsto \tilde H^\t_{\L_n}(\om_{\L_n})$ such that for every $\t\in \mathcal K$, $\tilde H^\t_{\L_n}= H^\t_{\L_n}$ $\pi_{\L_n}$-almost surely. \\

\noindent{\it Example 1 (continued)}:
For all $x\neq 0$, the function $\theta\mapsto \phi(x)$ in \eqref{strauss} is lower semicontinuous which ensures that the associated energy $\t\mapsto H^\t_{\L_n}$ for the Strauss model is lower semicontinuous as well. Changing the value of $\phi$  at the discontinuity points, it is easy to obtain an  upper semicontinuous version of the energy function which is $\pi_{\L_n}$-almost surely equal  to $H^\t_{\L_n}$. Therefore {\bf[Argmax]} holds for the Strauss model.

 \medskip

 \noindent{\it Example 2 (continued)}: {\bf [Argmax]} also holds for the Lennard-Jones model since the function $\theta\mapsto \phi(x)$ in \eqref{lennard} is continuous for any $x\neq 0$.

\begin{lemme}\label{MLEexistence} 
Under {\bf [Existence]} and {\bf [Argmax]}, the MLE of $\theta$ exists, i.e. there exists at least one $\hat\t_n\in\mathcal K$ such that 
$K_n^{\hat\delta_n, \hat\t_n}(\om^*_{\L_n}) =\min_{\t\in\mathcal K} K_n^{\hat\delta_n,\t}(\om^*_{\L_n})$. 
\end{lemme}

\begin{proof}
It is sufficient to show that  $\theta\mapsto K_n^{\hat\delta_n,\t}(\om^*_{\L_n})$ is lower semicontinuous. As a direct consequence of {\bf [Argmax]}, $\t\mapsto H^\t_{\L_n}(\om^*_{\L_n})$ is lower semicontinuous. For  the finite volume pressure, this comes from {\bf [Argmax]} and the Fatou's lemma since
\begin{align*}
\liminf_{\t'\mapsto \t} Z_{\L_n}^{\hat\delta_n,\t'} & = \liminf_{\t'\mapsto \t} \int e^{-\tilde H^{\t'}_{\L_n}(\om_{\L_n})} \1_{\Oi^{\hat\delta_n}}(\om_\L) \pi_{\L_n}(d\om_{\L_n})\\
& \ge   \int  \liminf_{\t'\mapsto \t} e^{-\tilde H^{\t'}_{\L_n}(\om_{\L_n})} \1_{\Oi^{\hat\delta_n}}(\om_\L) \pi_{\L_n}(d\om_{\L_n})\\
& \ge  \int e^{-\tilde H^{\t}_{\L_n}(\om_{\L_n})}\1_{\Oi^{\hat\delta_n}}(\om_\L) \pi_{\L_n}(d\om_{\L_n})=  Z_{\L_n}^{\hat\delta_n,\t}.
\end{align*} 
\end{proof}

\section{Consistency of the MLE for general Gibbs interactions}\label{general section}

  We detail in this section  the minimal assumptions on the parametric family of Gibbs measure  that imply the strong consistency of the MLE. 
All of them are fulfilled by the examples of Section~\ref{pair interaction}.  

The first ones, {\bf [Existence]} and  {\bf [Argmax]} introduced in the previous section, imply the existence of a Gibbs measure for any $\delta\in I$ and $\t\in\T$, and the existence of the MLE.

We assume the following mild assumptions on the family of energies. They gather the stability of the energy function, a standard stationary decomposition in terms of mean energy per unit volume and  they deal with  boundary  effects.   We put $\G=\cup_{\delta\in I, \theta\in\T} \G^{\delta,\t}$.

\bigskip

\noindent {\bf [Stability]}: For any compact set $\mathcal K\subset\Theta$, there exists a constant $\kappa\ge 0$ such that for any $\L$, any $\t\in\mathcal K$ and any $\om\in\O_T$
\begin{equation}\label{eq:stability} H_\L^\t(\om_\L) \ge -\kappa N_\L(\om).\end{equation}

\bigskip

\noindent {\bf [MeanEnergy]}: There exist measurable functions $H^\theta_0$ and  $\partial H^\t_n$ from $\O_T$ to $\R$ such that for all $n\ge 1$ and all $\theta\in\T$
\begin{equation}\label{dec}
H^\t_{\L_n}=\sum_{k\in\I_n} H^\t_0\circ \tau_{-k} + \partial H^\t_{\L_n},
\end{equation}
where for all $P\in\G$ and all $\theta\in\T$ the {\it mean energy} of $P$ defined by 
\begin{equation}\label{def mean energy} H^\t(P)=E_P(H_0^\t)\end{equation}
is finite and  for any compact set $\mathcal K\subset\Theta$, for $P$-almost every $\om\in\O_T$, 
$$\lim_{n\to\infty} \frac{1}{|\L_n|} \sup_{\t\in\mathcal K} \left|\partial H^\t_{\L_n}(\om)\right| =0.$$

\bigskip

\noindent {\bf [Boundary]}: For all $P\in\G$, for any compact set $\mathcal K\subset\Theta$ and for $P$-almost every $\om\in\O_T$ 
 $$ \lim_{n\to\infty} \frac{1}{|\L_n|} \sup_{\t\in\mathcal K}  \left| H^\t_{\L_n}(\om_{\L_n})-H^\t_{\L_n}(\om)\right| =0.$$\\

To prove the consistency of the MLE, we need to control the asymptotic behaviour of the finite volume pressure and the specific energy in \eqref{contrastf}. The following assumption deals with the latter.  \\

\noindent {\bf [Regularity]}: 
For any $P$ in $\G$, for any compact set $\mathcal K\subset\Theta$, there exists a function $g$ from $\R_+$ to $\R_+$ with $\lim_{r\to 0} g(r)=0$  such that for any $\t\in\K$ and any $r>0$
\begin{equation}\label{ineq regener} 
 E_P\left(\sup_{
\begin{subarray}{c}
\t'\in\K\\
|\t-\t'|\le r
\end{subarray}} 
\left| H_0^\t-H_0^{\t'}\right| \right) \le g(r).
\end{equation}
Moreover there exists a bounded subset $\Lambda_0$ such that for  any $\eta>0$ and any $\t_0\in\mathcal K$,  there exists a finite subset $\NNN(\t_0)\subset B(\t_0,\eta)\cap\T$ and $r(\t_0)>0$ such that $r(\t_0)<\eta$ and for any $\Lambda\supset \Lambda_0$,
\begin{equation}\label{regenerstab} 
\max_{\t\in\NNN(\t_0)} \inf_{\t'\in B(\t_0,r(\t_0))\cap \K}\inf_{\om_\L\in \Oi^{\delta_{\min}}} \left(\frac{H_{\L}^{\t}(\om_{\L})-H_{\L}^{\t'}(\om_{\L})}{N_{\L}(\om_{\L})}\right)\ge -g(r(\t_0)).\end{equation}

\bigskip

Assumption \eqref{ineq regener} implies in particular that the mean energy $H^\t(P)$ is continuous with respect to $\t$. Let us note that this does not imply that the specific energy or $\t\mapsto H_0^\t$ is continuous with respect to $\t$. As  a counterexample, the Strauss model in Example~1, or more generally any model from  Sections~\ref{FRhardcore} and \ref{FRsanshardcore}, has a discontinuous specific energy but the mean energy is continuous, which is verified in the proof of Theorem~\ref{th FR}. Assumption \eqref{regenerstab} is related to the regularity of the coefficient $\kappa$ in {\bf[Stability]} with respect to the parameter $\theta$.

Finally the following assumption relates the pressure and the specific entropy defined below,  through the variational principle. The variational principle is conjectured to hold for all models of statistical mechanics even if it is not proved in such a general setting. 
 It is established for stable and finite-range interactions in \cite{DerStu} and for infinite-range pair potentials that are regular and non-integrably divergent at the origin  in \cite{georgii94b,georgii94}.  

 For any $\delta\in I$ and $\t\in\T$,  the {\it specific entropy} of  $P^{\delta,\theta}\in\G^{\delta,\t}$  is the limit
\begin{equation}\label{defentropy} \I(P^{\delta,\t})=\lim_{n\to \infty} \frac{1}{|\L_n|} \I_{\L_n}(P^{\delta,\t})\end{equation}
where $\I_{\L_n}(P^{\delta,\t})= \int \ln f_{\L_n}^{\delta,\t}(\om_{\L_n})P_{\L_n}^{\delta,\t}(d\om)$  is the finite volume entropy of $P^{\delta,\t}$. Note that this limit always exists, see \cite{georgii}.\\

\noindent {\bf [VarPrin]}:  For any $\delta\in I$ and any $\t\in\T$, the {\it pressure} defined as the following limit exists and is finite
\begin{equation}\label{defpressure} p(\delta,\t)=\lim_{n\to\infty} \frac{1}{|\L_n|} \ln(Z^{\delta,\t}_{\L_n}).\end{equation}
 Moreover  the map $\delta\mapsto p(\delta,\t)$ is right continuous on $ I$, i.e. $\lim_{\delta'\to\delta,\delta'\geq\delta} p(\delta',\theta)=p(\delta,\theta).$
In addition, for any $\delta,\delta'$  in $ I$  with $\delta'\geq\delta$, for any  $\t,\t'$ in $\T$, for all  $P^{\delta',\t'}\in\G^{\delta',\t'}$ the inequality
\begin{equation}\label{ineq}
p(\delta,\t)\ge -\I(P^{\delta',\t'})-H^\t(P^{\delta',\t'})
\end{equation}

holds. If $(\delta',\t')=(\delta^*,\t^*)$, then  equality holds in \eqref{ineq} if and only if  $(\delta, \t)=(\delta^*, \t^*)$.\\

The formulation of the variational principle presented here is slightly weaker than the classical one in statistical mechanics  where the inequality (\ref{ineq}) is required for any probability measure $P$ on $\O_T$, and not only for $P^{\delta',\t'}\in\G^{\delta',\t'}$. Note that  the right continuity of the pressure holds trivially when $\delta^*$ is known, i.e. $ I=\{\delta^*\}$. Let us also remark that the last part in {\bf [VarPrin]} concerning $(\delta^*,\theta^*)$ is just an identifiability assumption. 

\medskip

We are now in position to state the consistency of the MLE.

\begin{theoreme}\label{Th1}
Under the assumptions  {\bf [Existence]}, {\bf [Argmax]}, {\bf [Stability]}, {\bf [MeanEnergy]}, {\bf [Boundary]}, {\bf [Regularity]}  and  {\bf [VarPrin]},  for any $(\delta^*,\t^*)\in I\times\mathcal K$ and any $P\in\G^{\delta^*,\t^*}$, the MLE  $(\hat\delta_n,\hat \t_n)$ is well-defined and converges $P$-almost surely to $(\delta^*,\t^*)$ when $n$ goes to infinity.
\end{theoreme}
 
\noindent{\it Example 1 (continued)}:   The Strauss and the hardcore Strauss models satisfy all the assumptions of Theorem \ref{Th1} which ensures the consistency of the MLE of $\delta^*, \beta^*$ and $R^*$.  The assumptions {\bf [Existence]} and {\bf [Argmax]} have already been discussed earlier and the stability assumption is obvious here since  $\phi\geq 0$. The assumptions {\bf [MeanEnergy]} and {\bf [VarPrin]} involve a variational characterisation of infinite volume Gibbs measures which is standard in  statistical mechanics \cite{DerStu}.  On the contrary, {\bf [Regularity]} is a technical assumption which is specific to the  study of consistency of the MLE. Its proof is given in Section \ref{details S1} when the hardcore parameter  $\delta^*$ is positive and  in Section \ref{details S2} when $\delta^*=0$. Note that the Strauss model belongs to the large class of  finite range pairwise interactions developed in details in Sections \ref{FRhardcore} and \ref{FRsanshardcore}.

 \medskip

 \noindent{\it Example 2 (continued)}: The Lennard-Jones model also satisfies the assumptions of Theorem \ref{Th1}. The assumptions {\bf [Existence]} and {\bf [Argmax]} have been checked in Section~\ref{param Gibbs}. The stability assumption is already known, see for instance Section 3.2 in \cite{Ruelle}, and is proved using the theory of positive type functions and Bochner's theorem. The assumptions  {\bf [MeanEnergy]} and {\bf [VarPrin]} are proved in \cite{georgii94b}, \cite{georgii94} which deal with the variational principle in the general setting of infinite range pairwise interactions. Again the {\bf [Regularity]} assumption is more specific and requires fine computations given in Section \ref{details LJ}. Note finally  that the Lennard Jones model belongs to the class of infinite range pairwise interactions studied in Section \ref{sectionLJ}.\\
 
 As illustrated in the two  examples above, all the assumptions of Theorem \ref{Th1}, excepted {\bf [Regularity]},  are natural and standard in statistical mechanics. They generally do not require much efforts to be checked since they are already proved  for many models in the literature. In fact, only Assumption {\bf [Regularity]} is really specific to the study of consistency of the MLE and needs a particular attention. Its proof strongly depends on the underlying model. Nonetheless, the techniques we develop in Sections \ref{details S1}, \ref{details S2} and \ref{details LJ}  in the cases of the Strauss, the hardcore Strauss and the Lennard Jones models (and their generalizations) provide tools and  strategies to deal with it. We think that they could be easily extended to the study of other models.

As a noticeable example stated in the following corollary, if the energy $H_\L^\theta$ depends  linearly on the parameter $\theta$, which is  called the exponential model, then the assumptions {\bf [Argmax]} and {\bf [Regularity]}  are automatically satisfied. Specifically, 
for any $1\le k\le p$, let $(H^k_\L)$ be a family of interaction energies satisfying  \eqref{compatible}. For any $\theta\in\Theta$ and any bounded set $\Lambda$ an exponential model takes the form 
\begin{equation}\label{energyexponentielle}
H_\L^{\theta}=\sum_{k=1}^{p} \theta_k H_\L^k.
\end{equation}
Let us introduce the following modification of {\bf [MeanEnergy]}.

\medskip

\noindent {\bf [MeanEnergy']}: The assumption {\bf [MeanEnergy]} holds with $H_0^\theta=\sum_{k=1}^{p} \theta_k H_0^k$ where $H_0^k$ comes from the decomposition \eqref{dec} applied to $H_{\L_n}^k$.

\begin{corollaire}\label{corollaireEM}
Assume that the family of interaction energies satisfy \eqref{energyexponentielle}  and that $\Theta$ is an open  subset of $\R^p$. Then under the assumptions  {\bf [Existence]}, {\bf [Stability]}, {\bf [MeanEnergy']}, {\bf [Boundary]} and {\bf [VarPrin]},  for any $(\delta^*,\t^*)\in I\times\mathcal K$ and any $P\in\G^{\delta^*,\t^*}$, the MLE  $(\hat\delta_n,\hat \t_n)$ is well-defined and converges $P$-almost surely to $(\delta^*,\t^*)$ when $n$ goes to infinity.
\end{corollaire}

\begin{remarque} 
 In  Theorem~\ref{Th1}, there is no topological assumption on the set  $\T$ except that $\T\neq\emptyset$.
 However, for identifiability reasons or in order to check  {\bf [Regularity]}, some assumptions on $\T$ will  often come out, depending on the model.  This is illustrated in Corollary~\ref{corollaireEM} where we assume $\T$  to be an open set. This assumption is usual and will still be used in the specific examples of the next section.  Nevertheless this hypothesis is  in general not necessary for consistency. For some models,  it is possible to check all assumptions of Theorem~\ref{Th1} when $\T$ is not open and $\theta^*$ is on the boundary of $\T$. As an example, it is not difficult to prove that the estimation of $\beta^*=0$ in the Strauss model (where $\delta^*=0$ and $R^*$ is known) is consistent, in which case $\t=\beta$, $\T=[0,\infty)$ and \eqref{energyexponentielle} is satisfied. This can be done as in the proof of Corollary~\ref{corollaireEM}, where  the assumption \eqref{regenerstab} in  {\bf [Regularity]} is verified by taking, for any $\eta>0$ and $\beta_0\geq 0$, $\NNN(\beta_0)=\{\beta_0+\eta/2\}$ and $r<\eta/2$. Note that the same conclusion is not true if in this model we try to estimate both $\beta^*$ and $R^*$ when $\beta^*=0$, because then $R^*$ is not identifiable. Finally, the study of consistency for parameters on the boundary of $\T$ is specific to the model at hand and it is difficult to draw general results in this case. 
   \end{remarque}

\section{Consistency for pairwise interactions and other examples}\label{pair 	interaction}

\subsection{Finite range  piecewise continuous pair potentials with hardcore}\label{FRhardcore}

In this section, we illustrate how the general result of Theorem~\ref{Th1} applies to the class of finite-range pairwise interactions with hardcore. 
The following framework gathers all potentials of this type described in  \cite{chiu2013}, \cite{B-IllPenSto08}, \cite{B-MolWaa03} and \cite{vanlieshout2000}, including the hardcore Strauss  model of Example 1 (detailed at the end of this section in Corollary~\ref{CorStraussHardcore}), the hardcore piecewise constant pairwise interaction model and  the Fiksel model.

Specifically, we consider conditional densities given by \eqref{parametric localdensity} with a non-vanishing hardcore parameter, i.e.  $\delta_{\min}>0$, and a Hamiltonian $H^\t_\L$  defined as in \eqref{defH} by  
 
\begin{equation}\label{defH2}
H^\t_\L (\om)=z\ N_\L(\om) + \sum_{\{x,y\}\in\om, \{x,y\}\cap \om_\L\neq\emptyset}  \phi^{\beta,R}(x-y),
\end{equation} 
where $z>0$, $\beta\in\R^p$ and $R=(R_1,\dots,R_q)$ is a multidimensional interaction parameter in $\R^q$ belonging to the subset   $\mathcal R=\{R\in\R^q,\ 0< R_1<\dots< R_q\}$. 
The pair potential interaction  $\phi^{\beta,R}$ is defined piecewise as follows.  For any $x\in\R^d$,
\begin{equation}\label{model FR}
\phi^{\beta,R} (x) = \begin{cases}  \1_{[0,R_1)}(|x|) \, \varphi_1^{\beta, R}(|x|)+\dots + \1_{(R_{q-1},R_q)}(|x|) \, \varphi_q^{\beta, R}(|x|), & \text{if } |x|\neq R_1,R_2,\ldots,R_q,\\
\min(\varphi_k^{\beta, R}(|x|),\varphi_{k+1}^{\beta, R}(|x|)), &\text{if } |x|=R_k \text{ for some } k,
\end{cases}
\end{equation}
with the convention $\varphi_{q+1}^{\beta,R}=0$.
 For any $k=1,\dots,q$,  we assume that there exists an open subset $\B$ of $\R^p$ such that 
 \begin{equation}\label{continu}
 (\beta,R,x)\mapsto \varphi_k^{\beta, R}(|x|) \text{ is continuous on } \B\times \mathcal R \times \R^d.
 \end{equation}
Note that from \eqref{model FR} and \eqref{continu}, the function $(\beta,R)\mapsto \phi^{\beta,R}(x)$ is lower semicontinuous which is crucial for the assumption {\bf [Argmax]}.

In this setting, the parameters to estimate are $\delta^*$ and $\theta=(z^*,\beta^*,R^*)$ that we assume belong to $I=[\delta_{\min},\infty)$ and $\T=\R_+\times \B \times \mathcal R$, respectively. For identifiability reasons, we further assume that
\begin{equation}\label{identifiabilite}
 R_1^*>\delta^* \quad \text{ and } \quad \lambda^d\left(x \in\R^d, \1_{[\delta^*,+\infty)}(|x|)\left(\phi^{\beta^*,R^*}(x)-\phi^{\beta,R}(x)\right)\neq 0\right)>0
\end{equation}
for any $(\beta,R)\neq(\beta^*,R^*)$.
This implies  that if $(\delta,\theta)\neq(\delta^*,\theta^*)$ then $\G^{\delta,\t}\cap\G^{\delta^*,\t^*}=\emptyset$ .

Notice that the  assumption $R_1^*>\delta^*$ is necessary to ensure identifiability of $R^*$ and $\delta^*$, but we do not need to add the constraint $R_1>\delta$ in the MLE optimisation  \eqref {def MLE}, as confirmed by the choice of  $\mathcal K$ in the theorem below. This is in particular important to agree with the setting of Lemma~\ref{MLE hardcore}, where it is  necessary that the optimisation  \eqref {def MLE} is carried out on  a cartesian product $ I\times\mathcal K$.

\begin{theoreme}\label{th FR}
Consider the Gibbs model defined by \eqref{parametric localdensity} with the Hamiltonian \eqref{defH2} and the pair potential \eqref{model FR}, under the assumptions $\delta_{\min}>0$, \eqref{continu} and \eqref{identifiabilite}. Let $\mathcal K$ be a compact subset of $\R_+ \times \B\times {\mathcal R}$ such that $(z^*,\beta^*,R^*)\in\mathcal K$. Then the MLE of $(\delta^*,z^*,\beta^*,R^*)$ given by \eqref {def MLE} is strongly consistent. 
\end{theoreme}

 \begin{corollaire}[Hardcore Strauss model]\label{CorStraussHardcore}
Consider the  hardcore Strauss model of Example 1 with parameters $\delta^*\geq\delta_{\min}$ where $\delta_{\min}>0$, $z^*>0$, $\beta^*\in\R$ and $R^*>\delta^*$.  Let $\mathcal K$ be a compact subset of $\R_+\times\R\times\R_+$ such that $(z^*,\beta^*,R^*)\in\mathcal K$. Then the MLE of $(\delta^*,z^*,\beta^*,R^*)$ given by \eqref {def MLE} is strongly consistent.  
\end{corollaire}

\begin{proof}
The hardcore Strauss model  corresponds to the Hamiltonian \eqref{defH2} and the pair potential \eqref{model FR} where   $q=p=1$ and $\varphi_1^{\beta,R}\equiv\beta$.  In this setting $\mathcal R=\R_+$. The assumptions \eqref{continu} and \eqref{identifiabilite} thus hold trivially true for $\B=\R$ and Theorem~\ref{th FR} applies.
\end{proof}

\subsection{Finite range  piecewise continuous pair potentials without hardcore}\label{FRsanshardcore}

We consider the same setting as in the previous section except that there is no hardcore,  i.e. $\delta^*=0$  and $\delta^*$ is known, meaning that it has not to be fitted, or equivalently $\delta_{\min}=\delta_{\max}=0$ implying  $ I=\{0\}$ in \eqref{def MLE}. 
The Hamiltonian $H^\t_\L$ is defined as in \eqref{defH2} and \eqref{model FR} with the same continuity assumptions  \eqref{continu}. The standard example we have in mind for this section is the Strauss model of Example~1 or any of the examples mentioned in the previous section without assuming the presence of a hardcore part.

Whereas this setting may appear to be simpler than the previous one from a statistical point of view (there is one parameter less to fit), it requires additional assumptions to ensure the existence of the model and to control its regularity.  This has been already pointed out for the Strauss model of Example 1 after the statement of {\bf[Existence]}. Consequently, we further assume  that the Hamiltonian $H^\t_\L$ satisfies {\bf[Stability]} and that for any $(\beta,R)$ the pair potential $\phi^{\beta,R}$ is superstable, which means that it is the sum of a stable pair potential plus a non  negative potential which is positive around zero (see \cite{Ruelle70}). The identifiability assumptions remain similar as in the previous section, that  \eqref{identifiabilite} with $\delta^*=0$. 

To handle the second part of {\bf[Regularity]}, we need a last hypothesis on the potential function. We assume that for any $(\beta_0,R_0)\in\B\times   \mathcal R$ and any $\eta>0$ there exists $\beta_1 \in B(\beta_0,\eta) \cap\B$ such that for any $k=1,\ldots,q$

\begin{equation}\label{localpositivity}
\varphi_k^{\beta_1,R_0}>\varphi_k^{\beta_0,R_0}.
\end{equation}

\begin{theoreme}\label{th FR2}
Consider the Gibbs model defined by \eqref{parametric localdensity} with the Hamiltonian \eqref{defH2} and the pair potential \eqref{model FR}, under the assumptions $\delta_{\min}=\delta_{\max}=0$, \eqref{continu} and \eqref{identifiabilite} where $\delta^*=0$. Assume further that the pair potential $\phi^{\beta,R}$ is superstable and \eqref{localpositivity}. Let $\mathcal K$ be a compact subset of $\R_+ \times \B\times {\mathcal R}$ such that $(z^*,\beta^*,R^*)\in\mathcal K$. Then the MLE of $(z^*,\beta^*,R^*)$ given by \eqref {def MLE} is strongly consistent. 
\end{theoreme}


This theorem can be easily applied to all standard finite range pairwise potentials (without hardcore) described in  \cite{chiu2013}, \cite{B-IllPenSto08}, \cite{B-MolWaa03} and \cite{vanlieshout2000}. In particular, condition~\eqref{localpositivity} turns out to be non-restrictive. Because of its central role in spatial statistics, we focus in the following corollary on the Strauss model of Example~1. 


\begin{corollaire}[Strauss model]\label{CorMS}
Consider the Strauss model of Example 1 with parameters $z^*>0$, $\beta^*>0$, $R^*>0$ and there is no hardcore, i.e. $\delta^*=0$ is known.  Let $\mathcal K$ be a compact subset of $\R_+^3$ such that $(z^*,\beta^*,R^*)\in\mathcal K$. Then the MLE of $(z^*,\beta^*,R^*)$ given by \eqref {def MLE} is strongly consistent.  
\end{corollaire}

\begin{proof}
Recall that the Strauss model  corresponds in  \eqref{model FR} to $q=p=1$ and $\varphi_1^{\beta,R}\equiv\beta$. As  in the proof of Corollary~\ref{CorStraussHardcore}, the assumptions \eqref{continu} and \eqref{identifiabilite} (where $\delta^*=0$) hold true with the choice $\B=\R_+$. Moreover the associated pair potential  $\phi^{\beta,R}$ is positive for any $\beta>0$ and $R>0$, which shows that  it is superstable and that {\bf[Stability]} holds true. Finally the assumption \eqref{localpositivity} is verified if we choose $\beta_1>\beta_0$. 
Therefore Theorem~\ref{th FR2} applies.
\end{proof}

\subsection{Infinite range pair potentials with a smooth parametrization}\label{sectionLJ}

In this section, we consider pairwise models  in the spirit of the Lennard-Jones model given in Example 2.  Our general setting concerns infinite range pair potentials without hardcore, i.e. $\delta_{\min}=\delta_{\max}=0$, that are uniformly regular and non-integrably divergent at the origin in the sense of Ruelle \cite{Ruelle70}. Specifically,  we assume that there exist two positive decreasing functions $\psi$ and $\chi$ from $\R_+$ to $\R$ and  $r_0>0$ with

\begin{equation*} \int_{r_0}^{+\infty} \psi(t)t^{d-1}dt<\infty, \qquad \int_0^{r_0} \chi(t)t^{d-1}dt=+\infty
\end{equation*}
such that for any parameter $\beta\in\B$, where $\B$ is an open set in $\R^p$, the function $\phi^\beta : \R^d\to \R$  satisfies for any $x\in\R^d$
\begin{equation}\label{integrabilitypsi} \phi^\beta(x)\ge \chi(|x|) \text{ whenever }|x|\le r_0 \quad \text{and} \quad |\phi^\beta(x)|\le \psi(|x|) \text{ whenever }|x|> r_0. 
\end{equation}

The Hamiltonian $H^\t_\L$ is defined as in \eqref{defH} by  
 
 \begin{equation}\label{hamilton_infinite}
H^\t_\L (\om)=z\ N_\L(\om) + \sum_{\{x,y\}\in\om, \{x,y\}\cap \om_\L\neq\emptyset}  \phi^{\beta}(x-y),\end{equation}
where $\theta=(z,\beta)$ is in  $\T=\R_+\times \B$.

We assume that $\phi^{\beta}$ is a symmetric function on $\Rd\backslash\{0\}$ and that for any $x\in\R^d\backslash\{0\}$, the map $\beta\mapsto \phi^\beta(x)$ is differentiable  on $\B$. Denoting by $\nabla\phi^\beta(x)$ its gradient,  we also assume that for any compact set $\K\subset\B$, for any $\beta\in\K$ and $x\in\R^d$ with $|x|>r_0$

\begin{equation}\label{gradientregular}
 |\nabla\phi^\beta(x)|\le \psi(|x|)
 \end{equation}
and 
\begin{equation}\label{goodexplosion}
\sup_{\beta\in\K} \sup_{x\in\R^d} \1_{|x|\le r_0} \sup_{\beta'\in\K} e^{-\phi^{\beta'}(x)}\max(\phi^\beta(x),|\nabla\phi^\beta(x)|) <\infty.
\end{equation}

In addition we suppose that for any compact set $\K\subset\B$ there exists an open set $U$ in $\R^p$ and a stable pair potential $\tilde \phi$ from $\R^d$ to $\R$ such that for any $u\in U$, any $x\in\R^d$ and any $\beta\in\mathcal K$

\begin{equation}\label{gradientdirection}
\nabla\phi^\beta(x).u \ge \tilde \phi(x),
\end{equation}
where $v.u$ denotes the scalar product of vectors $u,v$ in $\R^p$.

Finally, for identifiability reasons, we assume that for any $\beta\neq \beta'$ in $\B$,   

\begin{equation}\label{identifiabiliteLJ}
 \lambda^d \left(x \in\R^d, \phi^\beta(x)\neq \phi^{\beta'}(x)\right)>0.
 \end{equation}

\begin{theoreme}\label{th IR}
Let $(\phi^\beta)_{\beta\in\B}$ be a family of  pair potentials which are uniformly regular, non-integrably divergent at the origin and satisfy assumptions \eqref{gradientregular}-\eqref{identifiabiliteLJ}. Let $\mathcal K$ be a compact subset of $\T$ such that $\theta^*=(z^*,\beta^*)$ belongs to $\K$. Then the MLE of $(z^*,\beta^*)$ given by \eqref {def MLE} is strongly consistent. 
\end{theoreme}

As a fundamental example, the following corollary focuses on the Lennard-Jones model of Example 2. Let the  triangle domain   $\mathcal T=\{(n,m)\in\R^2, d<m<n\}$. 

\begin{corollaire}[Lennard-Jones model]
Consider the Lennard-Jones model of Example 2 given by \eqref{hamilton_infinite} and \eqref{lennard} with parameters $z>0$ and $\beta=(A,B,n,m)\in \B$ where $\B=\R_+ \times \R\times \mathcal T$. Let $\K$ be a compact subset of $\R_+\times \B$ such that $(z^*,\beta^*)\in\K$. Then the MLE of $(z^*,\beta^*)$ given by \eqref {def MLE} is strongly consistent. 
\end{corollaire}

\begin{proof}
Denote by $\K_{\mathcal T}$ the projection of $\K$ onto $\mathcal T$, i.e. $(n,m)\in\K_{\mathcal T}$ if and only if there exists $z>0$, $A>0$ and $B$ such that $(z,A,B,n,m)\in\K$.
It is easy to find $r_0>0$ and two positive constants $c_0$, $c_1$ such that $\phi^\beta$ given by \eqref{lennard} is uniformly regular and non-integrably divergent at the origin for $r_0$, $\psi(t)=c_0 t^{-m_0}$ and $\chi(t)=t^{-n_0}$ with $d<m_0<m$ and $d<n_0<n$ for all $(n,m)\in \K_{\mathcal T}$. 

The map $\beta \mapsto \phi^\beta$ is clearly differentiable on $\B$ with 

$$ \nabla \phi^\beta(x)=\left(|x|^{-n},|x|^{-m},-A\ln(|x|)|x|^{-n},B\ln(|x|)|x|^{-m}\right).$$

In adjusting the constant $c_0$ in the definition of $\psi$, we show that assumption \eqref{gradientregular} holds. Assumptions \eqref{goodexplosion} and \eqref{identifiabiliteLJ} are obvious. It remains to show \eqref{gradientdirection}. Let $u_0$ be the vector $(1,1,1,1)$ in $\R^4$ and $u$ any vector in the open ball $B(u_0,1/2)$ in $\R^4$. Denoting $\K_\B$ the projection of $\K$ onto $\B$, we find that for any $\beta\in\K_\B$ and $x\in\R^d$ 
\begin{equation*}
\nabla\phi^\beta(x).u \ge \frac{1}{2}\1_{[0,1]}(|x|)\Big(|x|^{-n}-|b_0\ln(|x|)||x|^{-m}\Big)-\frac{3}{2}\1_{(1,+\infty)}(|x|)(a_1+|b_1|)\ln(|x|)|x|^{-m}
\end{equation*}
where $a_1$, $b_0$ and $b_1$ are chosen so that for any $A$ and $B$, $A<a_1$ and $b_0<B<b_1$. There exists $\delta>0$ such that for all $(n,m)\in \K_{\mathcal T}$, $m<n-\delta$, whereby 
\begin{equation*}
\nabla\phi^\beta(x).u \ge c \1_{[0,1]}(|x|)|x|^{-n_0} - c \1_{(1,+\infty)}(|x|) \ln(|x|)|x|^{-m_0}
\end{equation*}
for some $c>0$. The right hand term of this inequality is regular non-integrably divergent at the origin. Therefore it is stable and \eqref{gradientdirection} is proved.
 \end{proof}

\medskip

We can also deduce from  Theorem~\ref{th IR} the same kind of result as in \cite{mase02} concerning exponential models. 
Note however that  the proof in \cite{mase02}  crucially relies on  the convexity of $\beta\mapsto\phi^\beta(x)$, whereas 
the following corollary is obtained by  different techniques.

\begin{corollaire}
Assume that in  \eqref{hamilton_infinite} the family of pair potentials $(\phi^\beta)_{\beta\in\B}$ is uniformly regular, non-integrably divergent at the origin, satisfies the identifiability condition \eqref{identifiabiliteLJ} and that 
$$\phi^\beta(x)=\sum_{i=1}^p \beta_i \phi_i(x)$$
where $\phi_1(x)\geq\chi(|x|)$ if $|x|<r_0$ while for $i\geq 2$, $\phi_i(x)=o(\phi_1(x))$ when $x\to 0$.
 Let $\mathcal K$ be a compact subset of $\R_+\times \B$ such that $\theta^*=(z^*,\beta^*)$ belongs to $\K$. Then the MLE of $(z^*,\beta^*)$ given by \eqref {def MLE} is strongly consistent. 
\end{corollaire}

\begin{proof}
In view of Theorem~\ref{th IR}, we just need to check  \eqref{gradientregular}-\eqref{gradientdirection}. Let $\beta=(\beta_1,\dots,\beta_p)\in \K$ and $\beta'=(\beta_1,\dots,\beta_{p-1},\beta'_p)\in \K$ with $\beta'_p\neq\beta_p$. Since $\phi^\beta$ is uniformly regular, we have $|\phi^\beta(x) - \phi^{\beta'}(x)|\leq 2\psi(|x|)$, meaning that $|\phi_p(x)| \leq c\psi(|x|)$ for some $c>0$.  We obtain likewise $|\phi_i(x)| \leq c\psi(|x|)$ for any $i=1,\dots,p$.
Consequenlty  $|\nabla\phi^\beta(x)|\leq  \sum_{i=1}^p |\phi_i(x)|\leq pc \Psi(|x|)$ and \eqref{gradientregular} is proved. 

The relation \eqref{goodexplosion} holds  trivially true by choosing $r_0$ sufficiently small. Finally, letting $U$ 
be any open set included in $\K$, we have $\nabla\phi^\beta(x).u =\phi^u(x)$ which is stable by assumption.  
 \end{proof}

\subsection{Examples of infinite-body interactions}\label{infinite body}

Our general result in Theorem~\ref{Th1} is also adapted to non-pairwise potential models. For instance, it is not difficult  to generalize the results of the three previous sections to the case of a finite-body interaction of order greater than two, as  for instance a triplet or quadruplet interaction. For brevity reasons, we do not include this generalization in the present paper. Instead, we give in this section examples arising from stochastic geometry which involve infinite-body interactions. Specifically, we focus on the area-interaction and the Quermass-interaction processes, but other models could have been considered as well.  

The area-interaction process  \cite{A-Baddeley95}, also called Widom-Rowlinson model in the statistical physics  community \cite{Widom70},  is probably the most popular model of infinite-body interaction. 
For $R\ge 0$, $\omega\in\Omega$ and $\L$ a bounded set in $\R^d$, we introduce the notation 
$$ \A^R(\omega_\L)=\lambda^d\left(\bigcup_{x\in\omega_\L} B(x,R)\right).$$
The Hamiltonian $H^{\t}_\L$ of the area-interaction process is defined for any $\om\in\Omega$ and any bounded set $\Lambda$ by
\begin{equation}\label{energyArea} H^\t_\Lambda(\omega)=zN_\L(\omega)+\beta\left(\A^R\left(\omega_{\Lambda\oplus B(0,2R)}\right)-\A^R\left(\omega_{\Lambda\oplus B(0,2R)\backslash \Lambda}\right)\right),
\end{equation}
where  the parameter $\t=(z,R,\beta)$ belongs to $\T=\R_+\times \R_+\times\R$ and the operator $\oplus$ stands for the Minkowski sum acting on the sets in $\R^d$. Note that  the unknown radius $R$ is part of the parameters and is estimated consistently as stated in the following proposition.

\begin{proposition}\label{prop Area}
Let $\HH=(H^\t)_\L$ be the family of area-interaction energies defined in \eqref{energyArea} and $\mathcal K$  a compact subset of $\T$ such that $\theta^*=(z^*,R^*,\beta^*)$ belongs to $\K$. Then the MLE of $\theta^*$ given by \eqref {def MLE} is strongly consistent. 
\end{proposition}

\medskip

The Quermass-interaction model is a generalization of the area-interaction process, where not only the volume of the union of balls is involved in the Hamiltonian but also the other Minkowski functionals. 
We denote by $M_k$,  $k=1\ldots d+1$, the $d+1$ Minkowski functionals in $\R^d$ and for short 
$$ \M_k^R(\omega_\L)=M_k\left(\bigcup_{x\in\omega_\L} B(x,R)\right),$$
where $R\geq 0$, $\om\in\Omega$ and $\L$ is a bounded subset of $\R^d$. Recall that for $d=2$, $M_1$ corresponds to  the Euler-Poincar\'e characteristic, $M_2$ is the perimeter and $M_3$ is the area. We refer to \cite{chiu2013} for more details about Minkowski functionals.

The Hamiltonian $H^{\t}_\L$ of the Quermass-interaction process is defined for any $\om\in\Omega$ and any bounded set $\Lambda$ by   
\begin{equation}\label{energyQuermass} H^\t_\Lambda(\omega)=zN_\L(\omega)+\sum_{k=1}^{d+1} \beta_k \left(\M_k^R\left(\omega_{\Lambda\oplus B(0,2R)}\right)-\M_k^R\left(\omega_{\Lambda\oplus B(0,2R)\backslash \Lambda}\right)\right)
\end{equation}
where $R\geq 0$, $z>0$ and $\beta=(\beta_1,\dots,\beta_{d+1})\in\R^{d+1}$.

This model has been introduced in  \cite{A-Kendall99}. Its existence on $\R^d$ has been solved  so far only when $d\leq 2$  in \cite{A-Der09}. Therefore we restrict the following study to the case $d\leq 2$.  
Moreover, we assume that $R$ is known and we only consider the MLE estimation of  $\t:=(z,\beta)$ in \eqref{energyQuermass}.  The reason is that we did not succeed to prove that \eqref{regenerstab}  in   {\bf [Regularity]} holds in presence of the Euler-Poincar\'e characteristic $M_1$ when $R$ is part of the unknown parameter.  
With the assumption that $R$ is known, the Quermass-interaction process  becomes an infinite-body interactions exponential model  and  Corollary~\ref{corollaireEM} applies. In this framework the set of parameters is $\Theta=\R_+ \times \R^{3}$.

\begin{proposition}\label{prop quermass}
Let $\HH=(H^\t)_\L$ be the family of Quermass-interaction energies defined in \eqref{energyQuermass} for $d=2$ and let $\mathcal K$ be a compact subset of  $\T$ such that $\theta^*=(z^*,\beta^*)$ belongs to $\K$. Then the MLE of $\theta^*$ given by \eqref {def MLE} is strongly consistent. 
\end{proposition}

\section{Proofs} \label{proofs}

 \subsection{Proof of Theorem~\ref{Th1}}

The proof is organised as follows. In a first step (Lemma \ref{consistency delta}), we show that the hardcore parameter $\hat\delta_n$ converges to $\delta^*$. So, according to Lemma~\ref{MLEexistence}, it remains to prove that the minimizer of $\theta \to K_n^{\hat\delta_n,\t}(\om^*_{\L_n})$ converges to $\theta^*$. This result is guaranteed via a general lemma on the convergence of  minimizers of possibly non regular contrast functions (Lemma \ref{propC}). The final step of the proof therefore consists in checking the assumptions of this lemma, namely
\begin{itemize}
\item the almost sure convergence of the contrast functions to a lower semicontinuous function admitting $\theta^*$ as a minimizer (assumptions $i)$, $ii)$, $iv)$ of Lemma \ref{propC}),
\item the control, when $n$ goes to infinity, of the infimum of  contrast functions evaluated on small balls with respect to the minimum of contrast functions evaluated only on a finite number of points (assumption $v)$). 
\end{itemize}

Note that  Lemma \ref{propC} and the verification of its assumptions in the setting of Theorem \ref{Th1} and for all models considered in Section~\ref{pair 	interaction} are the main contributions of the present paper.

Let us now enter into the details of the proof. Let $P$ be a probability measure in $\G^{\delta^*,\t^*}$ and let $\om^*$ be a realisation of $P$.   If $P$ is not ergodic, it can be represented as the mixture of ergodic stationary Gibbs measures, see \cite{B-Pre76}. Therefore the proof of the consistency of the MLE reduces to the case  when $P$ is ergodic, which is assumed henceforth.

Let us start by proving the convergence of the hardcore parameter.

\begin{lemme}\label{consistency delta}
The MLE  $\hat\delta_n$ converges  $P$-almost surely to $\delta^*$ when $n$ goes to infinity.
\end{lemme}

\begin{proof}
From Lemma~\ref{MLE hardcore}, $\hat \delta_n =\min( \tilde\delta_n,\delta_{\max})$ where $\tilde\delta_n$ is given by \eqref{delta tilde}. Clearly $\tilde\delta_n$ is a decreasing sequence and $\tilde\delta_n\geq\delta^*$. So it remains to prove that, for any $\delta'>\delta^*$, $\tilde \delta_n$ is smaller than $\delta'$ for $n$ large enough. By the DLR equations (\ref{DLR}) and the definition of local densities (\ref{parametric localdensity}), there exists a bounded set $\L$ such that
$$ P(\text{there exist } \; x,y\in\omega_\Lambda \text{ such that } \delta^*<|x-y|<\delta')>0.$$

Thanks to the ergodic theorem, this implies that
$$  P(\text{there exist } n\ge 1 \text{ and }  x,y\in\omega_{\Lambda_n} \text{ such that } \delta^*<|x-y|<\delta')=1$$
which proves the expected result. 
\end{proof}

\medskip

According to Lemmas~\ref{MLEexistence} and \ref{consistency delta}, we have now to prove  the consistency of $\hat\t_n$ where 
$$ \hat\t_n =\text{argmin}_{\t\in\mathcal K} K_n^{\hat\delta_n,\t}(\om^*_{\L_n}).$$
The strong consistency relies on a minimum contrast function result stated in the next lemma, where neither the contrast function nor  its limit  need to be continuous with respect to $\t$.
 
  \begin{lemme}\label{propC}
For any $\t\in\T$,  let $(h^\t_n)_{n\ge 1}$ be a family of parametric measurable functions from $\O_T$ to $\R$. Let $\mathcal K$ be a compact subset of   $\T$ and  let $P$ be a probability measure on $\O_T$. We assume that 
  
  {\it i)} For any $\t\in\T$, $h^\t_n$ converges $P$-almost surely to a finite real number (denoted by $h^\t$) when $n$ goes to infinity.
  
{\it ii)} The function $\t\mapsto h^\t$ admits a unique minimum over $\T$: 
$$ \t^*=\text{argmin}_{\t\in\T} h^\t,$$  
and $\theta^*\in \K$.

{\it iii)} For $P$-almost every $\om\in\O_T$ and  for $n$ sufficiently large, $\t\mapsto h_n^\t$ admits an infimum over $\mathcal K$ attained by at least one element $\hat\t_n(\om)$.

{\it iv)} The function $\t\mapsto h^\t$ is lower semicontinuous on $\T$ (i.e. for any $\t\in\T$, $\liminf_{\t'\mapsto \t} h^{\t'} \ge h^\t$).

{\it v)} There exists a function $g_0$ from $\R_+$ to $\R_+$ satisfying $$ \lim_{x\mapsto 0} g_0(x)=0$$
and such that for any $\eps>0$ and for any $\t\in \mathcal \K$,  there exists a finite subset $\NNN(\t)$ included in $B(\t,g_0(\eps))\cap\T$ and $r(\t)>0$ such that

\begin{equation}\label{controlinf} P\left(\limsup_{n\to\infty} \left( \min_{\t'\in\NNN(\t)} h^{\t'}_n-\inf_{\t'\in B(\t,r(\t))\cap \K} h^{\t'}_n \right)   \ge \eps\right)=0.\end{equation}

Then the sequence $\hat\t_n=\text{argmin}_{\t\in \mathcal K} h_n(\theta)$ converges $P$-almost surely to $\t^*$ when $n$ goes to infinity.
\end{lemme}

\begin{proof}
Let $\eps>0$ and $A_\eps$ be the following  set

$$ A_\eps=\left\{ \t\in\K\oplus  \bar B(0,g_0(\eps)) \text{ such that } h^\t\le  h^{\t^*}+2\eps\right\},$$
 where $g_0$ comes from {\it v)}    and $\bar B(0,g_0(\eps))$ denotes the closure of the  ball $B(0,g_0(\eps))$.
 By {\it iv)} the function $\t\mapsto h^\t$ is lower semicontinuous and so the set $A_\eps$ is a compact set. Moreover by {\it ii)},

\begin{equation}\label{inter}
\{\t^*\}= \bigcap_{\eps>0} A_{\eps}.
\end{equation}

We deduce that the diameter of $A_\eps$ goes to zero when $\eps$ goes to zero. 
  Let us consider the set $B_\eps:=A_\eps\oplus B(0,g_0(\eps)+\eps)$. Clearly its diameter tends to zero as well.
So to prove that  $(\hat\t_n)$ converges $P$-almost surely to $\t^*$, it is sufficient to show that for any $\eps>0$ 

\begin{equation}\label{limsup}
 P\left(\limsup_{n\to \infty}  \left\{ \hat\t_n \in B_\eps^c \right\} \right)=0,
 \end{equation}
where $B_\eps^c = \K \setminus (B_\eps\cap \K)$.

Since $\K$ is compact, there exists a finite sequence $(\t_i)_{1\le i \le N}$ in $\K$ and $\{r_i\}_{1\le i \le N}$ in $\R_+$, where $r_i=r(\t_i)$ comes from {\it v)}, such that the union of balls $B(\t_i,r_i)$ covers $\K$. Note that if \eqref{controlinf} holds for some $r$, then it holds for any $r'\leq r$. So  without loss of generality we can assume  $r_i\leq \eps$ for all $i=1,\dots,N$.  Let $\I_\eps$ be the subset of indexes $i\in\{1,\ldots,N\}$ such that the ball $B(\t_i,r_i)$ intersects the set $B^c_\eps$. We denote by $\NNN_i$ the subset $\NNN(\t_i)$ in assumption {\it v)} and by  $\NNN_\eps$ the union of $\NNN_i$ for $i\in I_\eps$. By definition and since for all $i=1,\dots,N$, $r_i\leq \eps$  and $\NNN_i\subset B(\t_i,g_0(\eps))$, we have

\begin{equation}\label{cover}
B^c_\eps\subset \bigcup_{i\in\I_\eps} B(\t_i,r_i),\quad  \NNN_\eps \subset \K\oplus  \bar B(0,g_0(\eps)) \quad \text{ and } \quad \NNN_\eps\cap A_\eps=\emptyset.
\end{equation}

The two latest properties imply that 
\begin{equation}\label{net}
\forall \t\in\NNN_{\eps}, \quad  h^\t>  h^{\t^*}+2\eps.
\end{equation}

Let us finally consider the following sequence of random variables:
 $$ \delta^\eps_n= \max_{i\in\I_\eps} \left( \min_{\t'\in\NNN_i} h^{\t'}_n-\inf_{\t'\in B(\t_i,r_i)\cap\K} h^{\t'}_n \right).$$
 Note that from \eqref{controlinf} and since $\I_\eps$ is finite
 \begin{equation}\label{delta}
 P \left(\limsup_{n\to \infty} \Big\{\delta_n^\eps\ge \eps \Big\} \right) =0.
\end{equation}
 
To prove (\ref{limsup}), we consider the following inclusions.
 \begin{align*}
\limsup_{n\to \infty} \left\{\hat \t_n \in B^c_\eps \right\} & \subset  \limsup_{n\to\infty} \left\{ \inf_{\t\in B^c_\eps} h_n^\t \le h_n^{\t^*} \right\}\\
 & \subset  \limsup_{n\to\infty} \left\{ \min_{\t\in\NNN_{\eps}} h_n^\t - h_n^{\t^*}\le \delta^\eps_n\right\}\\
&\subset   \limsup_{n\to\infty} \Big\{\delta_n^\eps \ge \eps\Big\}\bigcup \left\{\min_{\t\in\NNN_{\eps}} h_n^\t- h_n^{\t^*}<\eps\right\}\\
& =    \limsup_{n\to\infty} \Big\{ \delta_n^\eps \ge \eps\Big\}\bigcup \limsup_{n\to\infty} \left\{ \min_{\t\in\NNN_{\eps}} h^\t_n - h_n^{\t^*} < \eps \right\}  \\
& =  \limsup_{n\to\infty} \Big\{ \delta_n^\eps \ge \eps\Big\}\bigcup \left\{ \liminf_{n\to \infty} \left(\min_{\t\in\NNN_{\eps}} h^\t_n - h_n^{\t^*}\right) \le  \eps \right\}  \\
& \subset  \limsup_{n\to\infty} \Big\{ \delta_n^\eps \ge \eps\Big\}\bigcup \left\{ \min_{\t\in\NNN_{\eps}} h^\t - h^{\t^*} \le \eps \right\}\bigcup \O_{\eps},  
\end{align*}
where $\O_{\eps}$ is the event where $h_n^\t$  does not converge to $h^\t$, when $n$ goes to infinity, for some $\t$ in $\NNN_{\eps}\cup\{\t^*\}$. Since $\NNN_{\eps}\cup\{\t^*\}$ is countable and by assumption {\it i)}, $P(\O_{\eps})=0$.
 Moreover, from \eqref{net}, the set $\left\{ \min_{\t\in\NNN_{\eps}} h^\t - h^{\t^*} \le \eps \right\}$ is empty.
Recalling (\ref{delta}), we thereby deduce  (\ref{limsup}) and the lemma is proved.
\end{proof}

\medskip

 Let us show that the assumptions {\it i)}, {\it ii)}, {\it iii)}, {\it iv)} and {\it v)} of Lemma~\ref{propC} hold for the family of contrast functions 
 $$h_n^\t(\om^*_{\L_n}) =  K_n^{\hat\delta_n,\t}(\om^*_{\L_n}) = \frac{\ln(Z^{\hat\delta_n,\t}_{\L_n})}{|\L_n|}+ \frac{H^{\t}_{\L_n}(\om^*_{\L_n})}{|\L_n|}.$$

To prove {\it i)},
  first note, from {\bf [MeanEnergy]},  {\bf [Boundary]} and the ergodic theorem, that for any $\t\in\T$ and  $P$-almost all $\om$ 
$$ \lim_{n\to\infty} \frac{1}{|\L_n|} H^\t_{\L_n}(\om_{\L_n})= E_P(H_0^\t)=H^\t(P).$$
Second,  recall from assumption {\bf [VarPrin]} that   for any $\delta\in I$ and $\t\in\T$, $\ln(Z^{\delta,\t}_{\L_n})/{|\L_n|}$ converges to $p(\delta,\t)$ and $\delta\mapsto p(\delta,\t)$ is right continuous.  On the other hand, from   Lemma~\ref{consistency delta}, $\hat\delta_n \to \delta^*$ almost surely with $\hat\delta_n > \delta^*$. 
Let $\t\in\T$, $\epsilon>0$ and $\delta^+>\delta^*$ such that $|p(\delta^*,\t) - p(\delta^+,\t)|<\epsilon$. Let $n$ be sufficiently large so that  $\hat\delta_n<\delta^+$,  $|\ln(Z^{\delta^+,\t}_{\L_n})/{|\L_n|} - p(\delta^+,\t)|<\epsilon$ and $|\ln(Z^{\delta^*,\t}_{\L_n})/{|\L_n|} - p(\delta^*,\t)|<\epsilon$. Then, since $\delta\mapsto Z^{\delta,\t}_{\L_n}$ is decreasing, we have 
$$\ln(Z^{\hat\delta_n,\t}_{\L_n})/{|\L_n|} - p(\delta^*,\t) \leq \ln(Z^{\delta^*,\t}_{\L_n})/{|\L_n|} - p(\delta^*,\t)<\epsilon$$ and $$p(\delta^*,\t) - \ln(Z^{\hat\delta_n,\t}_{\L_n})/{|\L_n|} \leq p(\delta^*,\t) - p(\delta^+,\t) + p(\delta^+,\t) - \ln(Z^{\delta^+,\t}_{\L_n})/{|\L_n|} < 2\epsilon,$$
 proving that   $\ln(Z^{\hat\delta_n,\t}_{\L_n})/{|\L_n|}$ converges to $p(\delta^*,\t)$ almost surely.

Hence for any $\t\in\T$ and for $P$-almost every $\om^*$  $$ \lim_{n\to \infty} h_n^\t(\om^*_{\L_n}) = h^{\t}$$
where 
$$h^\t=p(\delta^*,\t)+H^\t(P)$$ 
which proves assumption {\it i)}.

To prove  Assumption {\it ii)}, note from  {\bf [VarPrin]} that for any $\t\in\T$, any  $P^{\delta^*,\t} \in \G^{\delta^*,\t}$ and any $P^{\delta^*,\t^*} \in \G^{\delta^*,\t^*}$
$$p(\delta^*,\t)\ge -\I(P^{\delta^*,\t^*})-H^\t(P^{\delta^*,\t^*})$$
while $ -\I(P^{\delta^*,\t^*}) = p(\delta^*,\t^*)+ H^\t(P^{\delta^*,\t^*})$
so that 
$$h^\t \geq h^{\t^*}.$$ 
From {\bf [VarPrin]}, the equality holds if and only if $\t=\t^*$ which proves Assumption {\it ii)}.

Assumption  {\it iii)} is given by Lemma~\ref{MLEexistence}.

 The function $\t\mapsto H^\t(P)$ is continuous thanks to \eqref{ineq regener}  in assumption {\bf [Regularity]}.
 By assumption {\bf [VarPrin]} we have for any $\theta$ and $\theta'$ in $\T$
 $$ p(\delta^*,\t')\ge p(\delta^*,\t)+ H^\t(P^{\delta^*,\t})-H^{\t'}(P^{\delta^*,\t}),$$
where $P^{\delta^*,\t}\in\G^{\delta^*,\t}$. By continuity of $\t\mapsto H^\t(P)$, it follows that $\t\mapsto p(\delta^*,\t)$ is lower semicontinuous. Therefore $\t\mapsto h^\t$ is lower  semicontinuous and  assumption {\it iv)} holds.

It remains to prove assumption {\it v)}. Let us start with some preliminary results.

Thanks to  {\bf [MeanEnergy]}, {\bf [Boundary]},  \eqref{ineq regener} in {\bf [Regularity]} and the ergodic Theorem, we have that, for any $\t\in\T$, any $r>0$ and $P$-almost every $\om$,

\begin{equation}\label{min1}
\limsup_{n\to\infty} \frac{1}{|\L_n|} \sup_{
\begin{subarray}{c}
\t'\in\K\\
|\t-\t'|<r
\end{subarray}
} \left|H_{\L_n}^\t(\om_{\L_n})-H_{\L_n}^{\t'}(\om_{\L_n})\right| \le g(r).
\end{equation}
This inequality allows us to control the variation of the infimum of the specific energy in the contrast functions. The following lemma deals with  the variation of the finite volume pressure.

\begin{lemme}\label{LemmeP} For any $\eta>0$ and any $\t_0\in\mathcal K$,  there exists a finite subset $\NNN(\t_0)\subset B(\t_0,\eta)\cap\T$ and $0<r(\t_0)<\eta$ such that for $n\ge 1$ and any $\delta\in I$
\begin{equation}\label{controlZ}  
\min_{\t\in\NNN(\t_0)} \frac{1}{|\L_n|} \ln(Z_{\L_n}^{\delta,\t}) - \inf_{\t\in B(\t_0,r(\t_0))\cap \K} \frac{1}{|\L_n|} \ln(Z_{\L_n}^{\delta,\t})\le e^{\kappa+1} g(r(\t_0)),
\end{equation}
where $g$ comes from assumption {\bf [Regularity]} and $\kappa$ from assumption {\bf [Stability]}.

\end{lemme}

\begin{proof} 
For any $\delta\in I$ and any $\t,\t'$ in $\T$ 
\begin{align}\label{lnExp}
\ln(Z_{\L_n}^{\delta,\t'})-\ln(Z_{\L_n}^{\delta,\t}) & =  \ln\left( \int \frac{1}{Z_{\L_n}^{\delta,\t}} e^{-H_{\L_n}^{\t'}(\om_{\L_n})} \1_{\Oi^\delta}(\om_{\L_n})\pi_{\L_n}(d\om_{\L_n})\right)\nonumber \\
& =  \ln\left( \int  e^{-H_{\L_n}^{\t'}(\om_{\L_n})+ H_{\L_n}^{\t}(\om_{\L_n})}  \1_{\Oi^\delta}(\om_{\L_n}) P_{\L_n}^{\delta,\t}(d\om_{\L_n})\right)\nonumber\\
& =  \ln E_{ P^{\delta,\t}_{\L_n}} (e^{-H_{\L_n}^{\t'}+ H_{\L_n}^{\t}}) 
\end{align}
where  $P_{\L_n}^{\delta,\t}$ is the probability measure $(1/Z^{\delta,\t}_{\L_n}) e^{-H_{\L_n}^{\t}} \1_{\Oi^\delta} \pi_{\L_n}$. By Jensen's inequality we have that

\begin{equation*}
\ln(Z_{\L_n}^{\delta,\t'})-\ln(Z_{\L_n}^{\delta,\t})  \ge  E_{ P^{\delta,\t}_{\L_n}}(H_{\L_n}^{\t}-H_{\L_n}^{\t'}).
\end{equation*}

For any $\eta>0$ and $\t_0$ in $\K$  we choose $\NNN(\t_0)$ and $r(\t_0)>0$ as in \eqref{regenerstab}. Denoting $[x]_+=\max(x,0)$ and writing for short $\NNN$ for $\NNN(\t_0)$ and $B$ for $B(\t_0,r(\t_0))\cap \K$, we obtain, for any $\delta\geq\delta_{\min}$,
\begin{align}\label{kk}
  \min_{\t\in\NNN} \frac{1}{|\L_n|} \ln(Z_{\L_n}^{\delta,\t})& - \inf_{\t\in B}  \frac{1}{|\L_n|} \ln(Z_{\L_n}^{\delta,\t}) \nonumber\\ 
  & =   \min_{\t\in\NNN} \sup_{\t'\in B} \frac{1}{|\L_n|} \left(\ln(Z_{\L_n}^{\delta,\t}) - \ln(Z_{\L_n}^{\delta,\t'})\right)\nonumber\\
   & \leq    \min_{\t\in\NNN} \sup_{\t'\in B} \frac{1}{|\L_n|} \left(E_{ P^{\delta,\t}_{\L_n}}(H_{\L_n}^{\t'}-H_{\L_n}^{\t})\right)\nonumber \\
& =  
 \frac{1}{|\L_n|}    \min_{\t\in\NNN} \sup_{\t'\in B}E_{ P^{\delta,\t}_{\L_n}}\left(\frac{H_{\L_n}^{\t'}-H_{\L_n}^{\t}}{N_{\L_n}}N_{\L_n}\right)\nonumber \\
  & \leq    \frac{1}{|\L_n|}    \min_{\t\in\NNN} \sup_{\t'\in B}\sup_{\om_{\L_n}\in \Oi^{\delta}} \left[\frac{H_{\L_n}^{\t'}(\om_{\L_n})-H_{\L_n}^{\t}(\om_{\L_n})}{N_{\L_n}(\om_{\L_n})}\right]_+ E_{ P^{\delta,\t}_{\L_n}}\left(N_{\L_n}\right)\nonumber \\
   & \leq    \frac{1}{|\L_n|} \left( \max_{\t\in\NNN} E_{ P^{\delta,\t}_{\L_n}}\left(N_{\L_n}\right) \right)  \left[ \min_{\t\in\NNN} \sup_{\t'\in B}\sup_{\om_{\L_n}\in \Oi^{\delta}} \left(\frac{H_{\L_n}^{\t'}(\om_{\L_n})-H_{\L_n}^{\t}(\om_{\L_n})}{N_{\L_n}(\om_{\L_n})}\right)\right]_+ \nonumber \\
     & =   
 \frac{1}{|\L_n|} \left( \max_{\t\in\NNN} E_{ P^{\delta,\t}_{\L_n}}\left(N_{\L_n}\right)  \right) \left[  -\max_{\t\in\NNN} \inf_{\t'\in B}\inf_{\om_{\L_n}\in \Oi^{\delta}} \left(\frac{H_{\L_n}^{\t}(\om_{\L_n})-H_{\L_n}^{\t'}(\om_{\L_n})}{N_{\L_n}(\om_{\L_n})}\right)\right]_+ \nonumber \\
  & \leq  \frac{1}{|\L_n|}  \left( \max_{\t\in\NNN} E_{ P^{\delta,\t}_{\L_n}}\left(N_{\L_n}\right)\right) g(r(\t_0)).
\end{align}

Let us control $E_{P^{\delta,\t}_{\L_n}}(N_{\L_n})$ by entropy inequalities.
By  definition of the entropy, the assumption {\bf [Stability]} and the standard inequality $Z_{\L_n}^{\delta,\t}\ge \pi_{\L_n}(\{\emptyset\})=e^{-|\L_n|}$, we find that
\begin{align}\label{eq1}
\I_{\L_n}( P^{\delta,\t}_{\L_n})& =  \int \ln\left(\frac{e^{-H_{\L_n}^\t} \1_{\Oi^\delta}(\om_{\L_n}) }{Z_{\L_n}^{\delta,\t}}\right)P^{\delta,\t}_{\L_n} \nonumber\\
& \leq  -\ln(Z_{\L_n}^{\delta,\t}) - \int H_{\L_n}^{\t}\, P^{\delta,\t}_{\L_n}\nonumber\\
& \le  |\L_n| +\kappa E_{P^{\delta,\t}_{\L_n}}(N_{\L_n}).
\end{align}
The entropy is also characterised by 
$$ \I_{\L_n}(P^{\delta,\t}_{\L_n})=\sup_{g} \left( \int g P^{\delta,\t}_{\L_n} -\ln\left(\int e^g \pi_{\L_n}\right)\right)$$
where the supremum is over all bounded measurable functions $g$, see for example \cite{varadhan1988}. Choosing $g=(\kappa+1)N_{\L_n} \1_{N_{\L_n}\le C}$ for any constant $C>0$, we get
\begin{align}\label{eq2}
\I_{\L_n}(P^{\delta,\t}_{\L_n}) &\ge (\kappa+1)E_{P^{\delta,\t}_{\L_n}}(N_{\L_n}\1_{N_{\L_n}\le C}) -\ln\left(\int e^{(\kappa+1)N_{\L_n}\1_{N_{\L_n}\le C}}\pi_{\L_n}\right) \nonumber \\
& \ge  (\kappa+1)E_{P^{\delta,\t}_{\L_n}}(N_{\L_n}\1_{N_{\L_n}\le C}) - (e^{\kappa+1}-1)|\L_n|.
\end{align}
Let $C$ goes to infinity and combine (\ref{eq1}) and (\ref{eq2}), we obtain that for any $(\delta,\t)$
\begin{equation}\label{controlN}E_{P^{\delta,\t}_{\L_n}}(N_{\L_n}) \le e^{\kappa+1}|\L_n|.\end{equation}
The inequalities (\ref{kk}) and \eqref{controlN} imply \eqref{controlZ}.
\end{proof}

We are now in position to prove {\it v)}. The function $g_0$ in {\it v)} depends on $g$ in \eqref{min1} and Lemma~\ref{LemmeP} as follows. For any $x>0$,
we choose $g_0(x)>0$ such that
$$  g(2\, g_0(x))+e^{\kappa+1} g(g_0(x))<x.$$
This choice is always possible since $g(u)$ tends to $0$ when $u$ goes to $0$.

Let $\eps>0$ and $\t_0\in\K$ and consider the finite subset $\NNN(\t_0)\subset B(\t_0,\eta)\cap\T$ and the positive number $r(\t_0)$, with $r(\t_0)<\eta$,  given by  Lemma~\ref{LemmeP} where $\eta=g_0(\eps)$. In the following, we write for short $\NNN$ for $\NNN(\t_0)$ and $B$ for $B(\t_0,r(\t_0))\cap \K$. Using \eqref{min1}, Lemma~\ref{LemmeP} and the definition of $g_0$, we have that for $P$-almost every $\om$,
\begin{align*}
& \limsup_{n\to\infty} \left(\min_{\t\in\NNN} h_n^\t(\om_{\L_n})- 
 \inf_{\t\in B} h_n^\t(\om_{\L_n})\right) \\
 &=  \limsup_{n\to\infty} \frac 1 {|\L_n|}\left(\min_{\t\in\NNN} \left(H_{\L_n}^\t(\om_{\L_n}) + \ln(Z_{\L_n}^{\hat\delta_n,\t})\right) -
 \inf_{\t\in B} \left(H_{\L_n}^\t(\om_{\L_n}) + \ln(Z_{\L_n}^{\hat\delta_n,\t})\right) \right) \\
 &\leq  \limsup_{n\to\infty} \frac 1 {|\L_n|}\left( \max_{\t\in\NNN} H_{\L_n}^\t(\om_{\L_n}) -  \inf_{\t\in B} H_{\L_n}^\t(\om_{\L_n}) + \min_{\t\in\NNN}  \ln(Z_{\L_n}^{\hat\delta_n,\t}) -  \inf_{\t\in B} \ln(Z_{\L_n}^{\hat\delta_n,\t}) \right) \\
&\leq \max_{\t\in\NNN}  \limsup_{n\to\infty} \frac 1 {|\L_n|}  \sup_{
\begin{subarray}{c}
\t'\in B
\end{subarray}
} \left|H_{\L_n}^\t(\om_{\L_n})-H_{\L_n}^{\t'}(\om_{\L_n})\right| +  \limsup_{n\to\infty} \frac 1 {|\L_n|} \left( \min_{\t\in\NNN}  \ln(Z_{\L_n}^{\hat\delta_n,\t}) -  \inf_{\t\in B} \ln(Z_{\L_n}^{\hat\delta_n,\t}) \right)\\
&\leq g(2\, g_0(\eps))+e^{\kappa+1} g(g_0(\eps)) < \eps,
 \end{align*}
which proves {\it v)}.

\subsection{Proof of Corollary~\ref{corollaireEM}}

To apply  Theorem \ref{Th1}, we have to prove that {\bf [Argmax]} and {\bf [Regularity]} hold true in the setting of exponential models.
Assumptions {\bf [Argmax]} and  the inequality \eqref{ineq regener}  in {\bf [Regularity]} are obviously satisfied since the energy $H_\L^\theta$  and the mean energy $H_0^\t$ are linear in $\theta$.

It remains to  check \eqref{regenerstab} in {\bf [Regularity]}. Let $\K$ be a compact subset of $\T$ and $\epsilon>0$ such that $\K\oplus B(0,\epsilon)\subset\T$. For any $\eta>0$ and $\theta_0\in\K$, we choose $r$ sufficiently small such that $\t:=(1+r/\epsilon)\theta_0\in B(\theta_0,\eta)\cap\T$ and we set $\mathcal N(\theta_0)=\{\theta\}$. Note that for any $\t'\in\B(\t_0,r)$, $\frac \epsilon r (\t-\t')\in B(\t_0,\epsilon)\subset \K\oplus B(0,\epsilon)$.
Then from \eqref{energyexponentielle} and the assumption  {\bf [Stability]}, for any $\theta'\in B(\theta_0,r)$, for any $\om_\L\in \Oi^{\delta_{\min}}$,
$$\frac \epsilon r (H_{\L}^{\t}(\om_{\L})-H_{\L}^{\t'}(\om_{\L}))=H_{\L}^{\epsilon(\t-\t')/r}(\om_{\L}) \geq -\kappa N_{\L}(\om_\L)$$
which implies $(H_{\L}^{\t}(\om_{\L})-H_{\L}^{\t'}(\om_{\L}))\geq -r \frac \kappa \epsilon N_{\L}(\om_\L)$ and thus \eqref{regenerstab}.

\subsection{Proof of Theorem~\ref{th FR}}   \label{details S1}
 We  check the assumptions of the more general Theorem~\ref{Th1} with the choice $\T=\mathcal Z \times \Upsilon$ where $\mathcal Z$ and $\Upsilon$ are bounded open subsets of $(0,\infty)$ and $\B\times\mathcal R$ respectively satisfying   $\mathcal Z \times \Upsilon\supset\mathcal K$. 
 We denote in the following  $\t=(z,\beta,R)$, where $z\in\mathcal Z$ and $(\beta,R)\in\Upsilon$.

For any $\delta\in I$ and $(\beta,R)\in\Upsilon$, the pairwise interaction $ \infty\1_{[0,\delta)}+\phi^{\beta,R}$  is superstable, finite range and bounded from above. The existence of Gibbs measures in this setting, i.e. assumption {\bf [Existence]},  is proved for example in \cite{Ruelle70}. Moreover, we can deduce  the following useful result from Corollary 2.9 in \cite{Ruelle70}: for any $P\in\G^{\delta,\theta}$, for any bounded set $\Delta$ and  any $c>0$, \begin{equation}\label{ruelle estimate}E_P\left(e^{cN _{\Delta}}\right)<\infty.\end{equation}

Superstability implies  {\bf [Stability]}, where the lower bound can be chosen uniformly over $\T$ since this set is bounded and using the hardcore property. 
The decomposition in {\bf [MeanEnergy]} can be done by 
\begin{equation}\label{Decpairwise}
H^\t_{\Lambda_n} (\om) = \sum_{k\in\I_n} H^\t_0\circ \tau_{-k}  (\om) + \partial H^\t_{\L_n}(\om)
\end{equation}
where 
\begin{equation}\label{H0pair}H^\t_0 (\om) = z\, N_{\Delta_0} (\om) + \sum_{\{x,y\}\in\om_{\Delta_0}} \phi^{\beta,R}(x-y) + \frac 12 \sum_{x \in\om_{\Delta_0}, y\in\om_{\Delta_0^c}} \phi^{\beta,R}(x-y)\end{equation}
 and 
 $$\partial H^\t_{\L_n}(\om) = \frac 12  \sum_{x \in\om_{\L_n}, y\in\om_{\L_n^c}} \phi^{\beta,R}(x-y).$$
 
 The finiteness of the mean energy given by \eqref{def mean energy}  is easily deduced from the finite range and boundedness of $ \phi^{\beta,R}$ along with \eqref{ruelle estimate} that  implies finite second order moments for $P$. On the other hand from \eqref{defH}  we have
$$H^\t_{\L_n}(\om) - H^\t_{\L_n}(\om_{\L_n})=  \sum_{x \in\om_{\L_n}, y\in\om_{\L_n^c}} \phi^{\beta,R}(x-y).$$
Therefore  {\bf [MeanEnergy]} and {\bf [Boundary]} reduce to prove that for $P$-almost every $\om\in\O_T$

\begin{equation}\label{aim}
\sup_{(\beta,R)\in\Upsilon} \frac 1 {|\L_n|} \sum_{x \in\om_{\L_n}, y\in\om_{\L_n^c}} \phi^{\beta,R}(x-y) \to 0.
\end{equation}

Recall the definition of $I_n=\{-n,\ldots, n-1\}^d$. Since $\phi^{\beta,R}$ has finite range and is continuous, there exist two constant $c$ and $\bar r$ large enough such that

\begin{align*} \sup_{(\beta,R)\in\Upsilon} \frac 1{|\L_n|} \sum_{x \in\om_{\L_n}, y\in\om_{\L_n^c}} |\phi^{\beta,R}(x-y)| & \le   \frac{c}{|\L_n|} \sum_{i\in I_n\backslash I_{n-1}}  N_{B(i,\bar r)\cap \L_n}(\om)N_{B(i,\bar r)\cap \L_n^c}(\om)\\
& \le  \frac{c}{|\L_n|} \sum_{i\in I_n\backslash I_{n-1}}  N_{B(i,\bar r)}^2(\om).
\end{align*}

From \eqref{ruelle estimate},  for any $P\in\G^{\delta,\theta}$,  $E_P(N_{B(i,\bar r)}^2)$ is finite and by stationarity $E_P(N_{B(i,\bar r)}^2)=E_P(N_{B(0,\bar r)}^2)$. The ergodic theorem thus implies that $P$ almost surely
$$ \frac{1}{|\L_n|} \sum_{i\in I_n}  N_{B(i,\bar r)}^2(\om) \to E_P(N_{B(0,\bar r)}^2)$$
which yields  \eqref{aim}.

 Assumption  {\bf [Argmax]} holds because $(\beta,R)\mapsto\phi^{\beta,R}(x)$ is lower semicontinuous by construction and the number of discontinuities are finite, so it is easy to change the sense of the brackets in \eqref{model FR} to obtain an upper semicontinuous version.

 \medskip
 
It remains to check   (a)   {\bf [Regularity]} and  (b) {\bf [VarPrin]}. 

 \medskip

(a) 
We start to prove  \eqref{ineq regener} in {\bf [Regularity]}.  Note from \eqref{H0pair} that 
\begin{multline}\label{regener1}
H^\t_0 (\om) - H^{\t'}_0 (\om)  =  (z-z')\, N_{\Delta_0} (\om) \\+ \sum_{\{x,y\}\in\om_{\Delta_0}} (\phi^{\beta,R}-\phi^{\beta',R'})(x-y) + \frac 12 \sum_{x \in\om_{\Delta_0}, y\in\om_{\Delta_0^c}} (\phi^{\beta,R}-\phi^{\beta',R'})(x-y).\end{multline}
If  $|\t-\t'|<r$, obviously there exists a positive function $g$ with $\lim_{s\to 0} g(s)=0$ such that 
\begin{equation}\label{ineqz} |z-z'|\, N_{\Delta_0} (\om) \leq g(r) N_{\Delta_0} (\om).\end{equation}
The two sums in the right hand side of \eqref{regener1} can be handled similarly. We give the details for the first one only, which reduces from \eqref{model FR} to consider the following generic  term for $P$-almost every $\omega$
\begin{align*}
 &\sum_{\{x,y\}\in\om_{\Delta_0}} |\1_{(R_{k-1},R_k)}(|x-y|) \, \varphi_k^{\beta,R}(|x-y|)-\1_{(R'_{k-1},R'_k)}(|x-y|) \, \varphi_k^{\beta',R'}(|x-y|)| \\ &\leq \sum_{\{x,y\}\in\om_{\Delta_0}}  |\varphi_k^{\beta,R}-\varphi_k^{\beta',R'}|(|x-y|) +  c \sum_{\{x,y\}\in\om_{\Delta_0}}  |\1_{(R_{k-1},R_k)}  -\1_{(R'_{k-1},R'_k)}|(|x-y|) 
\end{align*}
where $c$ is a positive constant that comes from the boundedness of $\varphi_k^{\beta,R}$.  Since the map $(\beta,R,x)\mapsto \varphi_k^{\beta, R}(|x|)$ is continuous it is uniformly continuous on any compact set and there exists a positive function $g$,  independent of $x$, $y$ and $\t$, such that $\lim_{s\to 0} g(s)=0$, and if $|\t-\t'|<r$ then 
\begin{equation}\label{term1} \sum_{\{x,y\}\in\om_{\Delta_0}}  |\varphi_k^{\beta,R}-\varphi_k^{\beta',R'}|(|x-y|) \leq  \sum_{\{x,y\}\in\om_{\Delta_0}} g(r) \leq g(r) N_{\Delta_0}^2(\om). \end{equation}

On the other hand, denoting $\Delta R_k$  the symmetric difference $(R_{k-1},R_k) \Delta (R'_{k-1},R'_k)$, we have for any $R\in\mathcal R$
\begin{align*} 
 E_P\left(\sup_{ \begin{subarray}{c}R'\in\mathcal R\\ |R-R'|\le r \end{subarray}}   \sum_{\{x,y\}\in\om_{\Delta_0}}  \1_{\Delta R_k}(|x-y|)\right) & = E_P\left( \sum_{x\in\om_{\Delta_0}}  \sup_{ \begin{subarray}{c}R'\in\mathcal R\\ |R-R'|\le r \end{subarray}}  \sum_{y\in\om_{\Delta_0}, y\neq x}   \1_{\Delta R_k}(|x-y|)\right)\\
& = E_P\left(\int_{\Delta_0} e^{-h^\t(x|\om)} \sup_{ \begin{subarray}{c}R'\in\mathcal R\\ |R-R'|\le r \end{subarray}}  \sum_{y\in\om_{\Delta_0}}   \1_{\Delta R_k}(|x-y|) \  dx\right)
\end{align*}
where the last equality comes from the Georgii-Nguyen-Zessin (GNZ) formula \citep{A-Geo76, A-NguZes79b} and $h^\t(x|\om)$ denotes the local energy needed to insert the point $x$ into the configuration $\om$, defined for  any $\L\ni x$ by  $ h^\t(x|\om)= H_\L^\t(\om\cup x) -H_\L^\t(\om)$. Applying again the GNZ formula we obtain 
 \begin{align} \label{term2}
& E_P\left(\sup_{ \begin{subarray}{c}R'\in\mathcal R\\ |R-R'|\le r \end{subarray}}   \sum_{\{x,y\}\in\om_{\Delta_0}}  \1_{\Delta R_k}(|x-y|)\right)\nonumber\\
& = E_P\left(\int_{\Delta_0^2} e^{-h^\t(y|\om)} e^{-h^\t(x|\om\cup\{y\})} \sup_{ \begin{subarray}{c}R'\in\mathcal R\\ |R-R'|\le r \end{subarray}}   \1_{\Delta R_k}(|x-y|)   dx dy\right)\nonumber \\
& \le  E_P\left(e^{cN _{\Delta_0 \oplus B(0,\bar r)}}\right) \int_{\Delta_0^2}\sup_{ \begin{subarray}{c}R'\in\mathcal R\\ |R-R'|\le r \end{subarray}} \1_{\Delta R_k}(|x-y|)  dx dy\nonumber \\
& \le E_P\left(e^{cN _{\Delta_0 \oplus B(0,\bar r)}}\right)  |\Delta_0| |\mathcal A_k(r)| 
\end{align}
where $c$ and $\bar r$ are two constants large enough and $\mathcal A_k(r)$ is the union of the two rings $\{y\in\R^d,\  R_{k-1}-r \leq |y| \leq R_{k-1}+r\}$ and $\{y,\  R_{k}-r \leq |y| \leq R_{k}+r\}$. Thanks to \eqref{ruelle estimate} and since $|\mathcal A_k(r)|\to 0$ as $r\to 0$, the upper bound  in \eqref{term2} tends to 0 as $r\to 0$. The same conclusion holds true in  \eqref{ineqz} and \eqref{term1} because  $E_P(N_{\Delta_0})<\infty$ and $E_P(N_{\Delta_0}^2)<\infty$. The combination of these results shows  \eqref{ineq regener}.

 Let us now prove  \eqref{regenerstab} in {\bf [Regularity]}. Let $\eta>0$ and $\t_0=(z_0,\beta_0,R_0)\in \mathcal K$ where $R_0=(R_{0,1},\dots,R_{0,q})$. 
We choose $r$ sufficiently small to ensure that 
\begin{itemize}
\item[(i)] $r<\eta$,
\item[(ii)] $B(\t_0,r)\subset \T$,
\item[(iii)] $2r<\inf_{k=1,\dots,q} |R_{0,k}-R_{0,k-1}|$,
\item[(iv)] for all $k=1,\dots,q$, if $\varphi_k^{\beta_0,R_0}(R_{0,k})\neq \varphi_{k+1}^{\beta_0,R_0}(R_{0,k})$, then for all $\t',\t''\in B(\t_0,r)$ and for all $x$ such that $| |x| - R_{0,k}|<r$,   we have 
$$\left(\varphi_k^{\beta',R'}(|x|)- \varphi_{k+1}^{\beta'',R''}(|x|)\right)  \left(\varphi_k^{\beta_0,R_0}(R_{0,k}) -   \varphi_{k+1}^{\beta_0,R_0}(R_{0,k})\right)>0,$$
 which is possible by continuity of $(\beta,R,x)\mapsto \varphi_k^{\beta, R}(|x|)$.
\end{itemize}
Then we fix $\mathcal N(\t_0)=\{(z_0,\beta_0, R)\}$ where $ R = ( R_1,\dots, R_q)$  is defined as follows : if $\varphi_k^{\beta_0,R_0}(R_{0,k})>\varphi_{k+1}^{\beta_0,R_0}(R_{0,k})$ then $ R_k = R_{0,k}+r$,  if $\varphi_k^{\beta_0,R_0}(R_{0,k})<\varphi_{k+1}^{\beta_0,R_0}(R_{0,k})$ then $ R_k = R_{0,k}-r$, and  if $\varphi_k^{\beta_0,R_0}(R_{0,k})=\varphi_{k+1}^{\beta_0,R_0}(R_{0,k})$ then $ R_k = R_{0,k}$. 

For $\t=(z_0,\beta_0, R)$ and $\t'\in B(\t_0,r)$, the proof of  \eqref{regenerstab}  amounts to control 
\begin{equation}\label{diffH}
\frac{H_{\L}^{\t}(\om_{\L})-H_{\L}^{\t'}(\om_{\L})}{N_{\L}(\om_{\L})}=(z_0-z')+ \frac 1 {N_{\L}(\om_{\L})} \sum_{\{x,y\}\in\om, \{x,y\}\cap \om_\L\neq\emptyset}\left( \phi^{ \beta_0,R}(x-y) - \phi^{\beta',R'}(x-y)\right).\end{equation}
The first term $(z_0-z')$ is greater than $-g(r)$ for a positive function $g$  with $\lim_{s\to 0} g(s)=0$. For the second term, let us introduce the middle points $\bar R_{0,k} = \frac 12 (R_{0,k-1}+R_{0,k})$ for $k=1,\dots,q-1$,  $\bar R_{0,0}=0$ and $\bar R_{0,q}= R_{0,q}+2r$. Note that from our choice of $r$, for any $k$ both $ R_k$ and $R'_k$ belong to $[\bar R_{0,k},\bar R_{0,k+1}]$. From \eqref{model FR}, we can write
\begin{align}\label{middle}
 \phi^{\beta_0,R}(x)  - \phi^{\beta',R'}(x)  = &\sum_{k=0}^q \1_{[\bar R_{0,k},\bar R_{0,k+1}]} (|x|) ( \phi^{\beta_0,R}(x)  - \phi^{\beta',R'}(x)) \nonumber\\
  = &\sum_{k=0}^q \Big( \1_{[\bar R_{0,k}, R_k)}(|x|) \varphi_k^{\beta_0,R} (|x|) + \1_{[ R_k,\bar R_{0,k+1}]}(|x|) \varphi_{k+1}^{\beta_0,R} (|x|) \nonumber\\
 & - 
 \1_{[\bar R_{0,k}, R'_k)}(|x|) \varphi_k^{\beta',R'} (|x|) - \1_{[ R'_k,\bar R_{0,k+1}]} (|x|) \varphi_{k+1}^{\beta',R'} (|x|) \Big).
 \end{align}
 If $\varphi_k^{\beta_0,R_0}(R_{0,k})>\varphi_{k+1}^{\beta_0,R_0}(R_{0,k})$, then $R_k>R'_k$ and the $k$-th term in the sum above writes
\begin{multline*} \1_{[\bar R_{0,k}, R'_k)}(|x|) \left( \varphi_k^{\beta_0,R} (|x|) -  \varphi_k^{\beta',R'} (|x|) \right) +  \1_{[ R'_k, R_k]}(|x|) \left( \varphi_k^{\beta_0,R} (|x|) -  \varphi_{k+1}^{\beta',R'} (|x|) \right) \\ +  \1_{[ R_k,\bar R_{0,k+1}]}(|x|) \left( \varphi_{k+1}^{\beta_0,R} (|x|) -  \varphi_{k+1}^{\beta',R'} (|x|)\right). \end{multline*}
From our choice of $r$, see (iv), the second term is always positive, and by uniform continuity of $(\beta,R,x)\mapsto \varphi_k^{\beta, R}(|x|)$, there exists a positive function $g$, independent of $x$, with $\lim_{s\to 0} g(s)=0$, such that the two remaining terms are greater than $-g(r)\1_{[0,\bar R_{0,q}]}(|x|)$. We thus obtain that the $k$-th term in \eqref{middle} is greater than $-g(r)\1_{[0,\bar R_{0,q}]}(|x|)$.

If $\varphi_k^{\beta_0,R_0}(R_{0,k})<\varphi_{k+1}^{\beta_0,R_0}(R_{0,k})$, we obtain the same lower bound by using the fact that  $R_k<R'_k$.
 If $\varphi_k^{\beta_0,R_0}(R_{0,k})=\varphi_{k+1}^{\beta_0,R_0}(R_{0,k})$, then by uniform continuity $\left| \varphi_k^{\beta,R} (|x|) -  \varphi_{k'}^{\beta',R'} (|y|)\right|< g(r)$ for any $k'=k, k+1$, any $|x|,|y|$ in $[\bar R_{0,k},\bar R_{0,k+1}]$ and any $\t,\t'$ in $B(\t_0,r)$,  and  the same lower bound holds for the the $k$-th term in \eqref{middle}.

Coming back to \eqref{diffH}, we deduce that  
\begin{align*}
\frac{H_{\L}^{\t}(\om_{\L})-H_{\L}^{\t'}(\om_{\L})}{N_{\L}(\om_{\L})} & > -g(r) - (q+1)  g(r)  \frac 1 {N_{\L}(\om_{\L})} \sum_{\{x,y\}\in\om, \{x,y\}\cap \om_\L\neq\emptyset} \1_{[0,\bar R_{0,q}]}(|x-y|).
 \end{align*}
For any $\om\in\Oi^{\delta_{\min}}$, there exists $c>0$ such that for any $x\in\om$, $\sum_{y\in\om, y\neq x} \1_{[0,\bar R_{0,q}]}(|x-y|)\leq c$ because of the hardcore distance $\delta_{\min}>0$. Consequently the lower bound above is greater than $-g(r)$, up to a positive constant, and this completes the proof.

\medskip

(b)  The variational principle in this setting  is proved in \cite{georgii94}, which implies \eqref
{ineq} with equality  if and only if  $P^{\delta',\t'}=P^{\delta,\t}$. When $(\delta',\t')=(\delta^*,\t^*)$,   $P^{\delta^*,\t^*}=P^{\delta,\t}$ is equivalent to 
$(\delta,\t)=(\delta^*,\t^*)$ from our identifiability assumption \eqref{identifiabilite}. Hence, in order to verify  {\bf [VarPrin]}, it remains to prove that the pressure  \eqref{defpressure} is right continuous in $\delta$.

For any $\delta>0$ and $\t\in\T$, from the rescaling $\om \to \delta \om$,
\begin{align*}
 Z_{\L_n}^{\delta,\t} & =  \int e^{-H_{\L_n}^{\t}(\om_{\L_n})} \1_{\Oi^\delta}(\om_{\L_n})\pi_{\L_n}(d\om) \\
&= e^{-|\Lambda_n|(1-1/\delta^d)} 
 \int e^{-H_{\L_n}^{\t}((\delta \om)_{\L_n})} \1_{\Oi^1}(\om_{\L_n})  \delta^{N_{\frac 1 \delta \L_n}(\om)} \pi_{\frac 1 \delta \L_n}(d\om)  \\
&=e^{-|\Lambda_n|(1-1/\delta^d)} \tilde Z_{\frac 1 \delta \L_n}^{\delta,\t},
\end{align*}
where $ \tilde Z_{\L}^{\delta,\t}  =  \int e^{-\tilde H_{\L}^{\delta,\t}(\om_{\L})} \1_{\Oi^1}(\om_{\L})\pi_{\L}(d\om_{\L})$ is the partition function associated to the Gibbs measure defined on $\Oi^1$ with energy function $\tilde H_{\L}^{\delta,\t}(\om_{\L}) = H_{\delta \L}^{\t}(\delta (\om_{\L})) - N_{\L} (\om) \ln(\delta)$.  Denoting  $\tilde p(\delta,\t)$  the pressure associated to the latter measure, we deduce 
\begin{equation}\label{pressure tilde}p(\delta,\t) = \frac 1 {\delta^d} \ \tilde p(\delta,\t)-1+1/\delta^d.\end{equation}
The energy function $\tilde H_{\L_n}^{\delta,\t}$ follows a decomposition like \eqref{dec} where 
\begin{equation*}
\tilde H^{\delta,\t}_0 (\om) = (z - \ln(\delta)) \, N_{\Delta_0} (\om)  + \sum_{\{x,y\}\in\om_{\delta \Delta_0}} \phi^{\beta,R}(\delta x- \delta y)  + \frac 12 \sum_{x \in\om_{\delta \Delta_0}, y\in\om_{\delta\Delta_0^c}} \phi^{\beta,R}(\delta x-\delta y).
\end{equation*}
By similar arguments as used earlier in (a) to prove   \eqref{ineq regener}, we deduce that $\delta\mapsto E_P( \tilde H^{\delta,\t}_0 )$ is continuous.
Moreover,  $\tilde H_{\L_n}^{\delta,\t}$ inherits the superstable, lower regular and regular properties of $\phi^{\beta,R}$, which shows from \cite{georgii94} that a variational principle holds, i.e. 
for any $(\delta,\theta)$ and $(\delta',\theta')$ 
$$ \tilde p(\delta',\t')\ge \tilde p(\delta,\t)+ \tilde H^{\delta,\t}(\tilde P^{\delta,\t})-\tilde H^{\delta',\t'}(\tilde P^{\delta,\t}),$$
where $\tilde P^{\delta,\t}$ denotes a Gibbs measure associated to  $\tilde H_{\L_n}^{\delta,\t}$. Note that contrary to \eqref{ineq}, this inequality is valid for any $\delta,\delta'$ because   $\tilde P^{\delta,\t}$ is supported on $\Oi^1$ for any $(\delta,\t)$. We therefore deduce that $\delta\mapsto \tilde p(\delta,\t)$ is lower semicontinuous, and so is $\delta\mapsto p(\delta,\t)$ from 
\eqref{pressure tilde}. Since $\delta\mapsto  Z_{\L_n}^{\delta,\t}$ is decreasing, $\delta\mapsto p(\delta,\t)$  turns out to be both decreasing and lower semicontinuous, which shows that it is right continuous.

\subsection{Proof of Theorem~\ref{th FR2}} \label{details S2}

 As in the previous section, it is sufficient to check the assumption of the main Theorem \ref{Th1}. The assumption {\bf[Stability]} is included in the assumptions of Theorem~\ref{th FR2} and has not to be showed. The proofs of assumptions {\bf [Existence]}, {\bf [MeanEnergy]}, {\bf [Boundary]}, {\bf [Argmax]} and the first part of {\bf [Regularity]}, i.e. \eqref{ineq regener}, are the same as in the proof of Theorem~\ref{th FR}. The variational principle is shown in \cite{DerStu} for any finite range stable pair potential. Since $ I=\{0\}$, the right continuity of the pressure is not required here, so Assumption {\bf [VarPrin]} follows. It remains to prove the second part of {\bf [Regularity]}, i.e.  \eqref{regenerstab}.

Let  $\eta>0$ and $\theta_0=(z_0,\beta_0,R_0)\in \K$ where $R_0=(R_{0,1},\ldots,R_{0,q})$. We denote by $\bar R$ the maximal value of $R_q$ when $\theta=(z,\beta,R)\in\K$.  We choose $r>0$ small enough such that

\begin{itemize}
\item[(i)] $r<\eta$,
\item[(ii)] $B(\t_0,r)\subset \T$,
\item[(iii)] $2r<\inf_{k=1,\dots,q} |R_{0,k}-R_{0,k-1}|$,

\item[(iv)] there exists $\beta_1\in \B$ such that for any $\theta'\in B(\theta_0,r)$, any $\theta=(z,\beta,R)\in B((z_0,\beta_1,R_0),\sqrt{qr})$, any $k=1,\ldots, q$ and any $x\in B(0,\bar R)$
$$ \varphi_k^{(\beta,R)}(|x|) \ge \varphi_k^{(\beta',R')}(|x|),$$
which is possible by continuity of $(\beta,R,x)\mapsto \varphi_k^{\beta, R}(|x|)$ and thanks to the local positivity assumption \eqref{localpositivity}. 

\item[(v)] for all $k=1,\dots,q$, if $\varphi_k^{\beta_0,R_0}(R_{0,k})\neq \varphi_{k+1}^{\beta_0,R_0}(R_{0,k})$, then for all $\t',\t''\in B(\t_0,r)\cup B((z_0,\beta_1,R_0),\sqrt{qr})$ and for all $x$ such that $| |x| - R_{0,k}|<r$,   we have 
$$\left(\varphi_k^{\beta',R'}(|x|)- \varphi_{k+1}^{\beta'',R''}(|x|)\right)  \left(\varphi_k^{\beta_0,R_0}(R_{0,k}) -   \varphi_{k+1}^{\beta_0,R_0}(R_{0,k})\right)>0,$$
 which is possible by continuity of $(\beta,R,x)\mapsto \varphi_k^{\beta, R}(|x|)$ and since $\beta_1$ can be chosen as close as we want to $\beta_0$.
\item[(vi)] for all $k=1,\dots,q$, if $\varphi_k^{\beta_0,R_0}(R_{0,k})= \varphi_{k+1}^{\beta_0,R_0}(R_{0,k})$, then  for all $\t'\in B(\t_0,r)$, $\t\in B((z_0,\beta_1,R_0),\sqrt{qr})$ and for all $x$ such that $| |x| - R_{0,k}|<r$, 
$$ \varphi_k^{(\beta,R)}(|x|) \ge \varphi_{k+1}^{(\beta',R')}(|x|) \quad  \text{ and } \quad \varphi_{k+1}^{(\beta,R)}(|x|) \ge \varphi_{k}^{(\beta',R')}(|x|), $$
which is possible for the same reasons as in  (iv). 

\end{itemize}

Now we choose $\mathcal N(\t_0)=\{(z_0,\beta_1, R)\}$ where $R=(R_1,R_2,\ldots,R_q)$ is as in the proof of Theorem \ref{th FR}, see \eqref{diffH} and before. Note that $(z_0,\beta_1,R)$ is in $B((z_0,\beta_1,R_0),\sqrt{qr})$ and that $\beta_1$ can be chosen sufficiently close to $\beta_0$ to ensure $\mathcal N(\t_0)\subset B(\t_0,\eta)$. Following the same calculus as in \eqref{middle}, we obtain that for any $\theta'=(z',\beta',R')\in B(\theta_0,r)$

$$ \phi^{\beta_1,R}(x) \ge \phi^{\beta',R'}(x).$$

We deduce that 

\begin{align*}
\frac{H_{\L}^{\t_0}(\om_{\L})-H_{\L}^{\t'}(\om_{\L})}{N_{\L}(\om_{\L})} & > z_0-z'>-r
 \end{align*}
and the second part of {\bf[Regularity]} is proved.

 \subsection{Proof of Theorem~\ref{th IR}}\label{details LJ}
 
 We check the assumptions of  Theorem \ref{Th1}. For any $\beta\in\B$ the pair potential $\phi^\beta$ is regular and non-integrably divergent at the origin. As it is well known (see \cite{Ruelle} for instance), $\phi^\beta$ is superstable and the existence of an associated Gibbs measures follows, i.e. {\bf [Existence]} holds true. Superstability implies stability which  shows assumption {\bf [Stability]}. Since the map $\beta\mapsto \phi^\beta$ is differentiable and so continuous, assumption {\bf [Argmax]} is obvious. The decomposition in {\bf [MeanEnergy]} is done as in  \eqref{Decpairwise}. In particular the mean energy \eqref{def mean energy} is 
 \begin{equation}\label{mean energy lennard}
 H^{\theta}(P)=E_P\left( z\, N_{\Delta_0}  + \sum_{\{x,y\}\in\om_{\Delta_0}} \phi^{\beta}(x-y) + \frac 12 \sum_{x \in\om_{\Delta_0}, y\in\om_{\Delta_0^c}} \phi^{\beta}(x-y)\right)\end{equation}
where $\theta=(z,\beta)$. To check  assumptions {\bf [MeanEnergy]} and {\bf [Boundary]}, we thus have to prove that  \eqref{mean energy lennard} is finite and that for any compact set $\K\subset\B$, 
 \begin{equation} \label{boundaryLJ}
\sup_{\beta\in\K} \frac 1 {|\L_n|} \sum_{x \in\om_{\L_n}, y\in\om_{\L_n^c}} \phi^{\beta}(x-y) \to 0.
\end{equation}
 For this purpose, we need the following Ruelle estimates.
   \begin{proposition}[\cite{Ruelle70}]\label{RuelleEstimate}
We define, for any $i\in\Z^d$, $ \psi_i=\psi(max(|i|-1,0),$
where the function $\psi$ comes from assumption \eqref{integrabilitypsi}. Then  $\sum_{i\in\Z^d} \psi_i<+\infty$ and for any $c>0$ 

$$E_P\left(e^{c\sum_{i\in\Z^d} \psi_i N_{\tau_i(\Delta_0)}}\right)<+\infty.$$

 \end{proposition}

Let us show \eqref{boundaryLJ}. Note that
\begin{equation*}
\sup_{\beta\in\K} \frac 1 {|\L_n|} \sum_{x \in\om_{\L_n}, y\in\om_{\L_n^c}} \phi^{\beta}(x-y) \le A_1+A_2
\end{equation*}
with
$$  A_1=  \frac 1 {|\L_n|} \sum_{
\begin{subarray}{c}
x \in\om_{\L_n\backslash \L_{n-n_0}}\\
 y\in\om, \,y\neq x\\
  |x-y|\le r_0
\end{subarray}  
  }  \sup_{\beta\in\K}|\phi^{\beta}|(x-y)  \qquad \text{and} \qquad  A_2= \frac 1 {|\L_n|}  \sum_{
\begin{subarray}{c}
x \in\om_{\L_n}\\
 y\in\om_{\L_n^c}\\
  |x-y|> r_0
\end{subarray}  
  }  \sup_{\beta\in\K}|\phi^{\beta}|(x-y),$$
 where $r_0$ comes from assumption    \eqref{integrabilitypsi} and $n_0$ is an integer greater than $r_0$. 

By the spatial ergodic theorem, $P$-almost surely
$$ \lim_{n\to \infty}  \frac 1 {|\L_n|} \sum_{
\begin{subarray}{c}
x \in\om_{\L_n}\\
 y\in\om,\,y\neq x\\
  |x-y|\le r_0
\end{subarray}  
  }  \sup_{\beta\in\K}|\phi^{\beta}|(x-y) = E\Bigg(\sum_{
\begin{subarray}{c}
x \in\om_{\Delta_0}\\
 y\in\om,\,y\neq x\\
  |x-y|\le r_0
\end{subarray}  }  \sup_{\beta\in\K}|\phi^{\beta}|(x-y) \Bigg).$$
Therefore, $A_1$ goes $P$-almost surely to $0$ if the expectation above is finite. By the GNZ equation, stationarity of $P$ and  assumption  \eqref{integrabilitypsi}
\begin{align}\label{borneesperance}
E\bigg(\sum_{
\begin{subarray}{c}
x \in\om_{\Delta_0}\\
 y\in\om,\,y\neq x\\
  |x-y|\le r_0
\end{subarray}  }  \sup_{\beta\in\K}|\phi^{\beta}|(x-y) \bigg) & = z^* E\bigg(e^{-\sum_{z\in\omega} \phi^{\beta^*}(z)}\sum_{
\begin{subarray}{c}
 y\in\om\\
  |y|\le r_0
\end{subarray}  }  \sup_{\beta\in\K}\phi^{\beta}(y) \bigg)\nonumber\\
& \le  z^*E\bigg( e^{\sum_{i\in\Z^d}  \psi_i N_{\tau_i(\Delta_0)}} \sum_
{\begin{subarray}{c}
 y\in\om\\
  |y|\le r_0
\end{subarray}  } e^{-\sum_{z\in\omega,|z|\leq r_0} \phi^{\beta^*}(z)}\sup_{\beta\in\K}\phi^{\beta}(y) \bigg)\nonumber\\
& \le  z^*E\bigg( e^{\sum_{i\in\Z^d}  \psi_i N_{\tau_i(\Delta_0)}} \sum_
{\begin{subarray}{c}
 y\in\om\\
  |y|\le r_0
\end{subarray}  }  e^{-\phi^{\beta^*}(y)}\sup_{\beta\in\K}\phi^{\beta}(y) \bigg)\nonumber\\
& \le  z^*E\bigg( e^{\sum_{i\in\Z^d}  \psi_i N_{\tau_i(\Delta_0)}} N_{B(0,r_0)} \bigg) \sup_{\beta\in\K} \sup_{y\in\B(0,r_0)}\sup_{\beta'\in\K} e^{-\phi^{\beta'}(y)} \phi^{\beta}(y)
\end{align}
which is finite by assumption \eqref{goodexplosion} and Proposition \ref{RuelleEstimate}. 
    
Regarding the term $A_2$, for any integer $K$ we have
 \begin{align}\label{calculus}
  A_2 & \le  \frac{1}{|\L_n|}\sum_{i\in \L_n, j\in\L_n^c} \psi_{|i-j|}N_{\tau_i(\Delta_0)}(\om)N_{\tau_j(\Delta_0)}(\om)\nonumber\\
  & \le   \frac{1}{|\L_n|} \sum_{i\in \L_n\backslash \L_{n-K}, j\in\Z^d} \psi_{|i-j|}N_{\tau_i(\Delta_0)}(\om)N_{\tau_j(\Delta_0)}(\om) +\frac{1}{|\L_n|} \sum_{
  \begin{subarray}{c}
i\in \L_{n-K}, j\in \Zd\\
  |j-i|\ge K
\end{subarray} } \psi_{|i-j|}N_{\tau_i(\Delta_0)}(\om)N_{\tau_j(\Delta_0)}(\om).
  \end{align}  
By Proposition \ref{RuelleEstimate}, $ N_{\Delta_0}\sum_{j\in\Zd} \psi_j N_{\tau_j(\Delta_0)}$ is integrable  
which implies thanks to the ergodic theorem that, when $n$ goes to infinity, the first term in \eqref{calculus} goes to zero and the second term goes to
$$  E\bigg(N_{\Delta_0}\sum_{j\in\Zd, |j|\ge K} \psi_j N_{\tau_j(\Delta_0)}\bigg).$$

Choosing $K$ large enough this last expectation can be made smaller than any positive value.  Therefore the term $A_2$ tends $P$-almost surely to zero as well and  \eqref{boundaryLJ} holds true.

By similar computations as above, using  the Ruelle estimates in Proposition~\ref{RuelleEstimate}, we obtain  that the mean energy \eqref{mean energy lennard} is finite. Assumptions {\bf[MeanEnergy]} and {\bf [Boundary]} are verified.

The existence and finiteness of the pressure in assumption {\bf[VarPrin]} is proved in Theorem~0.2 of \cite{Ruelle70}. The variational principle for such Gibbs measures is proved in \cite{georgii94b,georgii94} and thanks to the identifiability assumption\eqref{identifiabiliteLJ} we get {\bf [VarPrin]}.

 It remains to prove {\bf [Regularity]}. Let us show \eqref{ineq regener}. For any compact set $\mathcal Z\times \K\subset\T$ with $\mathcal Z\subset\R_+$ and $\K\subset\B$, for any $r>0$,
\begin{align*}
 E_P\bigg(\sup_{
\begin{subarray}{c}
\t'\in\mathcal Z\times\K\\
|\t-\t'|\le r
\end{subarray}} 
\left| H_0^\t-H_0^{\t'}\right| \bigg) &\le  rE_P(N_{\Delta_0})+ E_P\bigg(\sum_{x\in\omega_{\Delta_0}} \sum_{y\in\omega\backslash\{x\}} \sup_{
\begin{subarray}{c}
\beta'\in\K\\
|\beta-\beta'|\le r
\end{subarray}}|\phi^\beta-\phi^{\beta'}|(x-y)\bigg)\\
& \le   rE_P(N_{\Delta_0})+ r E_P\bigg(\sum_{x\in\omega_{\Delta_0}} \sum_{y\in\omega\backslash\{x\}} \sup_{\beta\in\K} |\nabla \phi^\beta|(x-y)\bigg).
\end{align*}
The proof of   \eqref{ineq regener} is completed if we show that the two expectations above are finite. This is clear for  the first one. We split the second one according to $|y|\leq r_0$ or $|y|>r_0$ where $r_0$ comes from assumption   \eqref{integrabilitypsi}. Thanks to similar computations as in \eqref{borneesperance} we get
\begin{multline*}
E_P\bigg(\sum_{x\in\omega_{\Delta_0}} \sum_{\begin{subarray}{c}
y\in\omega\backslash\{x\}\\
|y|\le r_0
\end{subarray}} \sup_{\beta\in\K} |\nabla \phi^\beta|(x-y)\bigg)\\ 
\leq z^*E\bigg( e^{\sum_{i\in\Z^d}  \psi_i N_{\tau_i(\Delta_0)}} N_{B(0,r_0)} \bigg)
\sup_{\beta\in\K} \sup_{y\in\B(0,r_0)}\sup_{\beta'\in\K} e^{-\phi^{\beta'}(y)}|\nabla\phi^{\beta}|(y)
\end{multline*}
which is finite by \eqref{goodexplosion} and Proposition~\ref{RuelleEstimate},  while from \eqref{gradientregular}
\begin{align*}
E_P\bigg(\sum_{x\in\omega_{\Delta_0}} \sum_{\begin{subarray}{c}
y\in\omega\backslash\{x\}\\
|y|> r_0
\end{subarray}} \sup_{\beta\in\K} |\nabla \phi^\beta|(x-y)\bigg)\leq E\bigg(N_{\Delta_0}\sum_{j\in\Zd} \psi_j N_{\tau_j(\Delta_0)}\bigg)
\end{align*}
which is also finite. This proves  \eqref{ineq regener}  in {\bf [Regularity]}.

 To prove \eqref{regenerstab}, let $\eta>0$ and $\theta_0=(z_0,\beta_0)\in\mathcal Z\times \K$. There exists  $0<\xi<\eta$ such that $B(\beta_0,\xi)\subset\B$. Consider the open set $U$ associated to $\K\oplus B(0,\xi)$ through the assumption \eqref{gradientdirection}.  Let us fix $u_0\in U$ and $\epsilon>0$ such that  $B(u_0,\epsilon)$ is a ball included in $U$. We choose $r$ sufficiently small to ensure that $r<\xi$ and $\beta:=\beta_0 + \frac r \epsilon u_0$ belongs to $B(\beta_0,\eta)$. By assumption, for any $\beta'\in B(\beta_0,r)$ and $x\in\R^d$,  there exists $t\in[0,1]$ such that
 $$ \phi^\beta(x)-\phi^{\beta'}(x)=\nabla \phi^{\beta'+t(\beta-\beta')}(x).(\beta-\beta').$$
 Note that the above choices ensure that $\beta'+t(\beta-\beta')$ belongs to $B(\beta_0,\xi)\subset\K\oplus B(0,\xi)$ and  that $\frac\epsilon r (\beta-\beta')$ belongs to  $B(u_0,\epsilon)\subset U$. Hence by assumption \eqref{gradientdirection}, for any $x\in\R^d$,
 $$\frac \epsilon r ( \phi^\beta(x)-\phi^{\beta'}(x)) \ge \tilde\phi(x)$$
 where  $\tilde\phi$ is a stable pair potential.
 We finally choose $\NNN(\theta_0)=\{\theta\}:=\{(z_0,\beta)\}$. Therefore for any bounded set $\L$, any $\theta'=(z',\beta')\in [z_0-r,z_0+r]\times  B(\beta_0,r)$ and any configuration $\omega$, denoting by $A>0$ the stability constant of $\tilde\phi$, we have
 \begin{align*}\label{ProofRegEner}
 H_\L^\theta(\omega_\L)-H_\L^{\theta'}(\omega_\L) \ge (z_0-z')N_\L(\omega_\L)+\sum_{\{x,y\}\in\omega_\L} \frac r \epsilon \,\tilde\phi(x-y)  \ge  -r\,(1+ A/\epsilon)N_\L(\omega_\L)
 \end{align*} 
  which proves \eqref{regenerstab} and concludes the proof.

 \subsection{Proof of Proposition \ref{prop Area}} \label{proofThArea}
 
We apply  Theorem \ref{Th1}. The local energy $h^\t(x|\omega)$ associated to the area-interaction process is,  for any $\L$ containing $x$,
 \begin{equation}\label{localenergyArea}
 h^\t(x|\omega)=H_\L^\t(\omega\cup x)-H_\L^\t(\omega)=z+\beta \lambda^d\left(B(x,R)\setminus \bigcup_{y\in\omega} B(y,R)\right).
\end{equation}
Since $\K$ is compact, $\bar z=\sup_{\t\in\K} z$, $\bar R=\sup_{\t\in\K} R$ and $\bar \beta=\sup_{\t\in\K} |\beta|$ are all finite and  $h^\t$ is uniformly bounded: 
$$ \sup_{x\in\R^d} \sup_{\t\in\K} \sup_{\omega\in\Omega} | h^\t(x|\omega)|\le \bar z+\bar\beta(2\bar R)^d.$$
This local stability property implies {\bf [Stability]}. 
Since $h^\t$ has also finite range $2\bar R$, an associated  Gibbs measure exists and the variational principle holds with the existence of the pressure, see \cite{B-Pre76}  and \cite{DerStu}. This gives {\bf [Existence]} and  {\bf [VarPrin]}. Since the application $\theta\mapsto H_\L^\theta(\omega)$ is continuous, the assumption {\bf [Argmax]} is obviously satisfied. As for  {\bf [MeanEnergy]} we suggest this decomposition (other one could have been considered)
$$ H^\theta_{\L_n}=\sum_{k\in\I_n} H_0^\t\circ\tau_{-k} +\partial H^\t_{\L_n}$$
with
\begin{equation}\label{meanenergyArea}
H_0^\t(\omega)=z N_{\Delta_0}(\omega) + \lambda^d\left(\Delta_0\cap\bigcup_{x\in\omega} B(x,R)\right).
\end{equation}
Clearly $E_P(H_0^\t)$ is finite. Moreover for any $n\ge 1$ 
$$  |\partial H^\t_{\L_n}| \le \bar\beta \lambda^d\left((\L_n\oplus B(0,\bar R))\setminus (\L_n \ominus B(0,\bar R))\right) \le 4d\bar \beta \bar R  (2n+2\bar R)^{d-1}  $$
which implies {\bf [MeanEnergy]}. Similar computations show that assumption {\bf [Boundary]} holds as well. 
Finally, Assumption {\bf [Regularity]}  is a consequence of the following uniform continuity of $H_0^\t$ coming from a simple geometric analysis: for any $\theta=(z,R,\beta)$, $\theta'=(z',R',\beta')$ in $\T$ (we assume $R\le R'$)
\begin{align*}
 |H_0^{\t}-H_0^{\t'}| &\le |H_0^{(z,R,\beta)}-H_0^{(z',R,\beta)}|+|H_0^{(z',R,\beta)}-H_0^{(z',R',\beta)}|+|H_0^{(z',R',\beta)}-H_0^{(z',R',\beta')}| \\
 & \le  |z-z'|N_{\Delta_0}+2d|\beta|(2R')^d|R-R'| N_{\Delta_0 \oplus B(0,R')}+|\beta-\beta'|.
 \end{align*}

\subsection{Proof of Proposition \ref{prop quermass}}

Since the Quermass-interaction model has an exponential form, it is sufficient to check the assumptions of Corollary~\ref{corollaireEM}. The existence of the model for any $z>0$ and any $\beta\in\R^3$ is proved in \cite{A-Der09} while stability of the interaction follows from  \cite{A-Kendall99}. The assumptions {\bf [MeanEnergy]} and {\bf [Boundary]} can be proved as in Section~\ref{proofThArea}, see \eqref{meanenergyArea}, thanks to the additivity of the Minkowski functionals. The assumption  {\bf [VarPrin]} follows from \cite{DerStu} since the interaction is stable and finite range. \\

\section*{Acknowledgments}
 The authors  are grateful to Shigeru Mase for supplying his unpublished manuscript \cite{mase02}.  They also thank Jean-François Coeurjolly for fruitful discussions and the anonymous referees for their suggestions which helped improve this paper. 
This work was partially supported by  the Labex CEMPI  (ANR-11-LABX-0007-01).

\bibliographystyle{abbrv}

\bibliography{bibaph}

\end{document}